\pdfoutput=1
\documentclass[hidelinks]{amsart}
\usepackage{mathtools}
\usepackage{amssymb}
\usepackage{xparse}
\usepackage{array}

\usepackage{bm,bbm}
\usepackage{mathrsfs}
\usepackage{stmaryrd}
\usepackage[scr=boondoxo]{mathalfa}

\usepackage[centering, footskip=6mm]{geometry}

\usepackage{xpatch}
\usepackage{mdframed}

\makeatletter
\xpatchcmd{\endmdframed}
  {\aftergroup\endmdf@trivlist\color@endgroup}
  {\endmdf@trivlist\color@endgroup\@doendpe}
  {}{}
\makeatother

\usepackage{iftex}

\ifpdf
\usepackage[hypertexnames=false]{hyperref}
\else
\usepackage[dvipdfmx,hypertexnames=false]{hyperref}
\fi
\usepackage{soul}

\usepackage[nameinlink]{cleveref}

\makeatletter
\renewcommand\th@plain{\slshape}
\makeatother

\numberwithin{equation}{section}

\newtheoremstyle{plain}
  {-\topsep}
  {}
  {\slshape}
  {}
  {\sffamily\bfseries}
  {.}
  {.5em}
  {}
\theoremstyle{plain}
\newtheorem{theorem}{Theorem}[section]
\newtheorem{corollary}[theorem]{Corollary}
\newtheorem{lemma}[theorem]{Lemma}
\newtheorem{proposition}[theorem]{Proposition}

\newtheorem*{claim*}{Claim}

\surroundwithmdframed[skipabove=\medskipamount,skipbelow=\medskipamount]{theorem}
\surroundwithmdframed[skipabove=\medskipamount,skipbelow=\medskipamount]{corollary}
\surroundwithmdframed[skipabove=\medskipamount,skipbelow=\medskipamount]{lemma}
\surroundwithmdframed[skipabove=\medskipamount,skipbelow=\medskipamount]{proposition}
\surroundwithmdframed[skipabove=\medskipamount,skipbelow=\medskipamount]{claim}
\surroundwithmdframed[skipabove=\medskipamount,skipbelow=\medskipamount]{claim*}

\newtheoremstyle{definition}
  {-\topsep}
  {}
  {\normalfont}
  {}
  {\sffamily\bfseries}
  {.}
  {.5em}
  {}
\theoremstyle{definition}

\newtheorem{remark}[theorem]{Remark}

\surroundwithmdframed[skipabove=\medskipamount,skipbelow=\medskipamount]{exercise}
\surroundwithmdframed[skipabove=\medskipamount,skipbelow=\medskipamount]{definition}
\surroundwithmdframed[skipabove=\medskipamount,skipbelow=\medskipamount]{notation}
\surroundwithmdframed[skipabove=\medskipamount,skipbelow=\medskipamount]{problem}
\surroundwithmdframed[skipabove=\medskipamount,skipbelow=\medskipamount]{question}
\surroundwithmdframed[skipabove=\medskipamount,skipbelow=\medskipamount]{note}
\surroundwithmdframed[skipabove=\medskipamount,skipbelow=\medskipamount]{remark}
\surroundwithmdframed[skipabove=\medskipamount,skipbelow=\medskipamount]{example}

\crefname{section}{Section}{Sections}

\crefname{theorem}{Theorem}{Theorems}
\crefname{corollary}{Corollary}{Corollaries}
\crefname{lemma}{Lemma}{Lemmas}
\crefname{proposition}{Proposition}{Propositions}
\crefname{claim}{Claim}{Claims}
\crefname{definition}{Definition}{Definitions}
\crefname{notation}{Notation}{Notations}
\crefname{problem}{Problem}{Problems}
\crefname{question}{Question}{Questions}
\crefname{note}{Note}{Notes}
\crefname{remark}{Remark}{Remarks}
\crefname{example}{Example}{Examples}

\crefname{enumi}{}{}
\crefname{enumii}{}{}
\crefname{enumiii}{}{}

\crefformat{equation}{{\upshape#2(#1)#3}}

\expandafter\def\csname ver@etex.sty\endcsname{3000/12/31}

\usepackage{autonum}

\makeatletter
\newcommand{\restore@Environment}[1]{%
  \AtBeginDocument{%
    \csletcs{#1*}{#1}%
    \csletcs{end#1*}{end#1}%
  }%
}
\forcsvlist\restore@Environment{alignat,equation,gather,multline,flalign,align}
\makeatother

\usepackage{enumitem}
\setlist{leftmargin=20pt}
\setlist[enumerate]{label=\textup{(\roman*)}}

\allowdisplaybreaks

\newcommand{\rev}{\overleftarrow}

\let\tmp\phi
\let\phi\varphi
\let\varphi\tmp
\let\tmp\epsilon
\let\epsilon\varepsilon
\let\varepsilon\tmp

\newcommand{\midmid}{\mathrel{}\middle|\mathrel{}}

\renewcommand{\mod}[1]{(\mathrm{mod}\ #1)}

\renewcommand{\and}{\quad\text{and}\quad}

\makeatletter
\NewDocumentCommand{\xsideset}{mmme{_^}}{%
\mathop{%
\settowidth{\dimen0}{$\m@th\displaystyle#3$}%
\dimen0=.5\dimen0
\settowidth{\dimen2}{$%
\m@th\displaystyle#3%
\IfValueT{#4}{_{#4}}%
\IfValueT{#5}{^{#5}}%
$}%
\dimen2=.5\dimen2
\advance\dimen2 -\dimen0
\sbox6{\scriptspace\z@$\displaystyle{\vphantom{#3}}#1$}
\sbox8{\scriptspace\z@$\displaystyle{\vphantom{#3}}#2$}
\ifdim\wd6>\dimen2 \kern\dimexpr\wd6-\dimen2\relax\fi
{%
\mathop{\llap{\copy6}{\displaystyle#3}\rlap{\copy8}}\limits%
\IfValueT{#4}{_{#4}}%
\IfValueT{#5}{^{#5}}%
}%
\ifdim\wd8>\dimen2 \kern\dimexpr\wd8-\dimen2\relax\fi
}%
}
\makeatother

\newcommand{\dsum}[1]{\xsideset{}{^{#1}}\sum}
\newcommand{\astsum}{\dsum{\smash{\ast}}}

\title[On Telhcirid's theorem on arithmetic progressions]
{On Telhcirid's theorem on arithmetic progressions}
\author[G. Bhowmik]{Gautami Bhowmik}
\author[Y. Suzuki]{Yuta Suzuki}
\keywords{Prime numbers, arithmetic progressions, reversed radix representation.}
\subjclass{%
Primary:
11A63; 
Secondary:
11N05, 
11N69. 
}

\begin{document}

\begin{abstract}
In this paper, we study the distribution of the digital reverses of prime numbers,
which we call the \textsl{reversed primes}.
We prove the infinitude of reversed primes in any arithmetic progression
satisfying straightforward necessary conditions provided the base is sufficiently large.
We indeed prove an effective Siegel--Walfisz type result for reversed primes,
which has a larger admissible level of modulus than the classical case.
\end{abstract}

\maketitle

\section{Introduction}
\label{sec:intro}

The investigation of various digital properties of 
integers has been popular in the last twenty years.
The problems include sums of digits or representations with missing digits.

Take $g\in\mathbb{Z}_{\ge2}$ to be the base of the digital representation.
We write the base-$g$ representation of $n\in\mathbb{N}$ as
\begin{equation}
\label{sec:notation:base_g_expansion}
n
=
\sum_{0\le i<N}n_{i}g^{i}
\quad\text{with}\quad
n_{0},\ldots,n_{N-1}\in\{0,\ldots,g-1\}
\ \text{and}\ 
n_{N-1}\neq0,
\end{equation}
where we say the integer $n$ or its representation is of the length $N$.

Notice that the distribution of primes with special digital properties are 
difficult problems in general,
though there are some known results on the infinitude of certain primes.
For example, Mauduit and Rivat~\cite{MauduitRivat:Gelfond}
proved that, for any base $g\ge2$, there are infinitely many primes
whose sum of digits satisfy a given congruence condition
and Maynard~\cite{Maynard:RestrictedDigitP}
showed that, for any base $g\ge10$, there are infinitely many primes
which do not have one given digit in their digital representation.
The aim of this paper is to show a new infinitude of primes with a special digital property. 

In this paper, we study the digital reverse of integers.
For the base-$g$ representation \cref{sec:notation:base_g_expansion} of a positive integer $n$, write
\begin{equation}
\label{def:rev}
\overleftarrow{n}
\coloneqq
\sum_{0\le i<N}n_{N-i-1}g^{i}
=
\sum_{0\le i<N}n_{i}g^{N-i-1},
\end{equation}
and call it the \textsl{digital reverse} of $n$.
In particular, we call the digital reverse $\rev{p}$ of a prime number 
$p$ a \textsl{reversed prime}.
Our aim is to obtain \textsl{Telhcirid's theorem on arithmetic progressions}.
More precisely, we prove, for any sufficiently large base $g$, the 
infinitude of reversed primes in a given arithmetic progression
except in the degenerate cases where the infinitude does not hold for 
trivial reasons.
We indeed have a quantitative version of this result,
an analog of the Siegel--Walfisz theorem for 
reversed primes, which could be called the \textsl{Zsiflaw--Legeis theorem}.

We avoid  the irregular behavior of the cardinality of  integers with special digital properties
by considering only the integers of a fixed length $N$.
Namely, for $N\ge1$, we count integers in the set
\begin{equation}
\label{def:G_N}
\mathscr{G}_{N}
\coloneqq
\{
g^{N-1}\le n<g^{N}
\mid
n\not\equiv0\ \mod{g}
\}
\end{equation}
of all positive integers of length $N$ with  non-zero lowest digit.
Note that there is no prime with zero as the lowest digit except $g$ itself.
Also, note that the operator $\rev{\ast}$ is an involution on $\mathscr{G}_{N}$.
For $a,q\in\mathbb{Z}$ with $q\ge1$, we then define
\[
\rev{\pi}_{\!\!N}(a,q)
\coloneqq
\sum_{\substack{
p\in\mathscr{G}_{N}\\
\rev{p}\equiv a\ \mod{q}
}}
1,
\]
which is the counting function for reversed primes of length $N$ in the 
arithmetic progression $a\ \mod{q}$, analogous to the 
classical $\pi(x,a,q)$ that counts the number of primes $\equiv a\ \mod{q}$ less than or equal to $x$.
There are few easy connections between the arithmetic property of a positive integer $n$ and $\rev{n}\ \mod{q}$.
First, for $n\in\mathscr{G}_{N}$, we have
\begin{equation}
\label{rev_congruence}
\rev{n}\equiv g^{N-1}n\ \mod{g^{2}-1}
\end{equation}
since $g^{-1}\equiv g\ \mod{g^{2}-1}$.
Second, for $\nu\in\mathbb{N}$, the residue $\rev{n}\ \mod{g^{\nu}}$ is determined
by the first $\nu$ leading digits of $n$.
However, besides these,
one may find no straightforward connection between $n$ and $\rev{n}\ \mod{q}$.
Therefore, by leaving the 
$\mod{(q,(g^{2}-1)g^{N})}$ condition as it is
and expecting the remaining to behave randomly,
one may expect
\[
\rev{\pi}_{\!\!N}(a,q)
=
\frac{(q,(g^{2}-1)g^{N})}{q}
\sum_{\substack{
p\in\mathscr{G}_{N}\\
\rev{p}\equiv a\ \mod{(q,(g^{2}-1)g^{N})}
}}1
+\rev{R}_{N}(a,q)
\]
with a small remainder term $\rev{R}_{N}(a,q)$
provided $q$ is not too large compared to $g^{N}$
so that $(q,g^{N})$ is determined by $q$ and $(q,g)$.
Similar to the classical Siegel--Walfisz type theorem for 
$\pi(x,a,q)$,
here we prove a pointwise estimate for the remainder term $\rev{R}_{N}(a,q)$.
Unfortunately, we have a strict restriction on the size of the base $g$.
\begin{theorem}
\label{thm:reversed_prime_AP}
For $g,a,q\in\mathbb{Z}$ with $g,q\ge2$,
there exists a constant $c\in(0,1)$ such that
\[
\overleftarrow{\pi}_{\!\!N}(a,q)
=
\frac{(q,(g^{2}-1)g^{N})}{q}
\sum_{\substack{
p\in\mathscr{G}_{N}\\
\rev{p}\equiv a\ \mod{(q,(g^{2}-1)g^{N})}
}}
1
+
O\biggl(
g^{N}
\exp\biggl(-c\cdot\frac{N}{\log q}\biggr)
\biggr)
\]
provided
\begin{equation}
\label{thm:reversed_prime_AP:g_cond}
\alpha_{g}
\coloneqq
\frac{\log(\frac{2}{\pi}\log\cot\frac{\pi}{2g}+\frac{1}{g\sin\frac{\pi}{2g}}+1)}{\log g}<\frac{1}{5}
\quad\text{or equivalently},\quad
g\ge31699
\end{equation}
and
\begin{equation}
\label{thm:reversed_prime_AP:d_cond}
q
\le
\exp\biggl(c\cdot\frac{N}{\log N}\biggr),
\end{equation}
where the constant $c$ and the implicit constant depend only on $g$ and are effectively computable.
\end{theorem}
For the monotonicity of $\alpha_{g}$ and the inequality $\alpha_{31699}<\frac{1}{5}$, see \cref{rem:alpha_g_monotonic}.

By \cref{rev_congruence},
if $(a,q,g^{2}-1)>1$,
a prime $p$ with $\rev{p}\equiv a\ \mod{(q,g^{2}-1)}$ is a prime divisor of $q$.
Also, the lowest digit of any reversed prime $\rev{p}$ is non-zero,
so there is no prime $p$ with $\rev{p}\equiv a\ \mod{(q,g)}$ if $g\mid (a,q)$.
Except for these degenerate cases,
where there are only finitely many reversed primes in the arithmetic progression $a\ \mod{q}$,
we can use some form of the effective prime number theorem
in arithmetic progressions in \cref{thm:reversed_prime_AP} to obtain the following asymptotic formula.

\begin{corollary}[Zsiflaw--Legeis theorem]
\label{cor:ZsiflawLegeis}
Under the same setting as in \cref{thm:reversed_prime_AP}, we have
\[
\overleftarrow{\pi}_{\!\!N}(a,q)
=
\frac{\rho_{g}(a,q)}{q}
\frac{g^{N}}{\log g^{N}}
\biggl(1+O\biggl(\frac{1}{N}\biggr)\biggr)
+
O(g^{N}\exp(-c\sqrt{N}))
\]
provided $g\ge31699$ and
\begin{equation}
\label{cor:ZsiflawLegeis:d_cond}
q
\le
\exp(c\sqrt{N}),
\end{equation}
where $c\in(0,1)$ is some constant, the function $\rho_{g}(a,q)$ is given by
\[
\rho_{g}(a,q)
\coloneqq
\left\{
\begin{array}{>{\displaystyle}cl}
\biggl(
1
-
\mathbbm{1}_{(q,g)\mid a}
\frac{(q,g)}{g}
\biggr)
\prod_{p\mid (q,g^{2}-1)}\biggl(\frac{p}{p-1}\biggr)
&\text{if $(a,q,g^{2}-1)=1$ and $g\nmid (a,q)$},\\[6mm]
0&\text{otherwise}
\end{array}
\right.
\]
and $c$ and the implicit constant depend only on $g$ and are effectively computable.
\end{corollary}

This asymptotic formula implies the following infinitude :
\begin{corollary}[Telhcirid's theorem on arithmetic progressions]
\label{cor:Telhcirid}
For $g,a,q\in\mathbb{Z}$ with
\[
g\ge31699,\quad
q\ge1,\quad
(a,q,g^{2}-1)=1,\quad
g\nmid(a,q),
\]
there are infinitely many primes $p$ such that $\rev{p}\equiv a\ \mod{q}$.
\end{corollary}

It is interesting to observe that
if $(a,q,g^{2}-1)=1$ and $g\nmid(a,q)$, \cref{cor:ZsiflawLegeis} implies
\[
\overleftarrow{\pi}_{\!\!N}(a,q)
\sim
\frac{\rho_{g}(a,q)}{a}
\frac{g^{N}}{\log g^{N}}
\quad\text{as $N\to\infty$}
\]
in the range
\[
q\le o(\exp(c\sqrt{N})N^{-1})
\quad\text{as $N\to\infty$}
\]
which is much wider
than the classical range $d\le N^{A}$ for the Siegel--Walfisz theorem with a constant $A$.
(Note that the main variable, say $x$, of the Siegel--Walfisz theorem
corresponds to $g^{N}$ here and so $\log x$ corresponds to $N$.)
Indeed, from a technical point of view,
we do not use the  Siegel--Walfisz theorem at all,
which makes \cref{thm:reversed_prime_AP} effective.
Note that when we deduce \cref{cor:ZsiflawLegeis} from \cref{thm:reversed_prime_AP},
we use the prime number theorem in arithmetic progressions with the modulus $(q,g^{2}-1)\ll_{g} 1$
and so there is no effect of the Siegel zeros provided  that we do not care about the 
dependence of the error term on the base $g$.
One may think of \cref{thm:reversed_prime_AP} as partial evidence that
primality and digital reversing are uncorrelated to each other.

There are some preceding results on the digital reverse of primes.
A positive integer $n$ is called a \textsl{palindrome} if $n=\rev{n}$.
If such an $n$ is a prime, e.g.\ 101, it is called a \textsl{palindromic prime}. 
The infinitude of palindromic primes is conjectured and
remains one of the difficult problems on the digital properties of primes.
In this context, Col~\cite{Col:PalindromeAP}
proved that the density of palindromic primes
among palindromes less than $x$ is $\ll(\log x)^{-1}$ as is expected,
and improved the earlier result of Banks, Hart and Sakata~\cite{BHS}.
We note that these papers contain Siegel--Walfisz type results for palindromes,
\cite[Corollary~4.5]{BHS} and \cite[Th\'{e}or\`{e}me~1]{Col:PalindromeAP}.

Further, as in more classical unsolved problems of prime numbers,
there exist partial results with almost primes instead of primes for their digital properties.
We call an integer $n$ an $r$-almost prime if $n$ has at most $r$ prime factors counted with multiplicity.
For example, Tuxanidy and Panario~\cite[Theorem~1.4]{TuxanidyPanario},
proved that there are infinitely many palindromic $6$-almost primes for any fixed base $g$,
which improved the result of Col~\cite[Corollaire~2] {Col:PalindromeAP}.
There is a similar but weaker conjecture on the digital property of primes.
A prime number $p$ is called a \textsl{reversible prime} if $\rev{p}$ is also a prime.
Again it is not known whether there are infinitely many reversible primes.
Note that a palindromic prime is automatically a reversible prime
and so the infinitude of reversible primes would follow from the infinitude of palindromic primes.
A partial result on the infinitude of reversible primes,
recently obtained by Dartyge, Martin, Rivat, Shparlinski and Swaenepoel~\cite{DMRSS:ReversiblePrime},
states that there are infinitely many integers $n$
such that both  $n$ and $\rev{n}$ are $8$-almost primes in base $g=2$.
The fact that their result on reversible primes is weaker than that of palindromic primes
is probably caused by the ineffectiveness of the two-dimensional sieve used there.

The method to prove \cref{thm:reversed_prime_AP}
is a mixture of the discrete circle method
used in preceding studies on the digital properties of integers
and the techniques on digital reverse available 
in \cite{Col:PalindromeAP,DMRSS:ReversiblePrime,MauduitRivat:carres,MauduitRivat:Gelfond,TuxanidyPanario}.
We mainly follow Maynard~\cite{Maynard:PnP} for the setting of the discrete circle method. 
The size restriction on the base $g$ is caused by the same reason as in \cite{Maynard:PnP},
i.e.\ the weakness of the $L^{1}$ bound (\cref{lem:L1_discrete_dev}) for small base $g$
(note that $C_{g}$ is a constant depending on $g$ in \cref{lem:L1_discrete_dev}
but its $M$-th power is not negligible).

\section{Notation}
\label{sec:notation}
Besides those introduced in the main body,
we use the following notations.

Throughout the paper,
$A$, $H$, $P$, $Q$, $R$, $X$ denote positive real numbers,
$t$, $u$, $x$, $z$, $\alpha$, $\beta$, $\delta$, $\epsilon$, $\eta$, $\theta$, $\kappa$, $\lambda$ denote real numbers,
$d$, $n$, $q$, $r$, $J$, $K$, $L$, $M$, $N$, $\nu$ denote positive integers,
$i$, $j$, $k$ denote non-negative integers and
$a$, $b$, $h$, $\ell$, $m$, $v$ denote integers.
The letter $p$ is reserved for prime numbers.
The letter $c$ is used for positive constants
which can take different values line by line.

For a real number $x$, we let $e(x)\coloneqq\exp(2\pi ix)$, while
$[x]$ denotes the integer part of $x$, i.e.\ the greatest integer $\le x$
and $\|x\|\coloneqq\min_{n\in\mathbb{Z}}|x-n|$ denotes the distance between $x$ and its nearest integer.

For a positive integer $q$, the symbol
\[
\astsum_{a\ \mod{q}}
\]
stands for the sum over reduced residues $\mod{q}$.

Throughout the paper, $g\ge2$ is an integer used as the base of radix representation.
For a positive integer $n$ with the base-$g$ representation \cref{sec:notation:base_g_expansion},
we define its digital reverse $\rev{n}$ by \cref{def:rev}.
We let
\begin{equation}
\label{def:C_g}
C_{g}
\coloneqq
\frac{2}{\pi}\log\cot\frac{\pi}{2g}
+\frac{1}{g\sin\frac{\pi}{2g}}
+1
\quad\text{and}\quad
\alpha_{g}
\coloneqq
\frac{\log C_{g}}{\log g}
=
\frac{\log(\frac{2}{\pi}\log\cot\frac{\pi}{2g}+\frac{1}{g\sin\frac{\pi}{2g}}+1)}{\log g}.
\end{equation}

For $N\in\mathbb{Z}_{\ge3}$ and $\alpha,\beta\in\mathbb{R}$, we use the exponential sums
\begin{equation}
F_{N}(\alpha,\beta)
\coloneqq
\sum_{n\in\mathscr{G}_{N}}
e(\alpha n+\beta\overleftarrow{n})
\quad\text{and}\quad
S_{N}(\alpha)
\coloneqq
\sum_{1\le p<g^{N}}e(\alpha p).
\end{equation}
For $N\in\mathbb{Z}_{\ge3}$, let us write
\[
\Phi_{N}(\alpha,\beta)
\coloneqq
\prod_{i=1}^{N-2}
\varphi(\alpha g^{i}+\beta g^{N-i-1})
\quad\text{with}\quad
\varphi(\alpha)
\coloneqq
\sum_{0\le n<g}e(\alpha n).
\]

The arithmetic function $\phi(n)$ stands for the Euler totient function.

For integers $n_{1},\ldots,n_{r}$,
we write $(n_{1},\ldots,n_{r})$ for the greatest common divisor of $n_{1},\ldots,n_{r}$.

For a logical formula $P$,
we write $\mathbbm{1}_{P}$
for the indicator function of $P$.

If a theorem or a lemma is stated
with the phrase ``where the implicit constant depends on $a,b,c,\ldots$'',
then every implicit constant in the corresponding proof
may also depend on $a,b,c,\ldots$ without being specifically mentioned.

\section{Auxiliary lemmas on the exponential sum with digital reverse}
\label{sec:lemmas_F}
We first prove some basic properties of the exponential sums $F_{N}(\alpha,\beta)$ and $\Phi_{N}(\alpha,\beta)$.

\begin{proposition}
\label{prop:swap}
We have $F_{N}(\alpha,\beta)=F_{N}(\beta,\alpha)$
and $\Phi_{N}(\alpha,\beta)=\Phi_{N}(\beta,\alpha)$.
\end{proposition}
\begin{proof}
The first equation is obvious since the digit reverse
is indeed an involution on $\mathscr{G}_{N}$.
The latter one is obtained by changing the variable
via $i\leadsto N-i-1$ as
\[
\Phi_{N}(\alpha,\beta)
=
\prod_{i=1}^{N-2}\varphi(\alpha g^{i}+\beta g^{N-i-1})
=
\prod_{i=1}^{N-2}\varphi(\beta g^{i}+\alpha g^{N-i-1})
=
\Phi_{N}(\beta,\alpha).
\]
Thus, we obtain the assertion.
\end{proof}

\begin{proposition}
\label{prop:product_formula}
For $N\in\mathbb{Z}_{\ge3}$, we have
\[
|F_{N}(\alpha,\beta)|
\le
g^{2}
|\Phi_{N}(\alpha,\beta)|
=
g^{2}
\prod_{i=1}^{N-2}|\varphi(\alpha g^{i}+\beta g^{N-i-1})|.
\]
\end{proposition}
\begin{proof}
By writing $n\in\mathscr{G}_{N}$ as
\[
n
=
\sum_{0\le i<N}n_{i}g^{i}
\quad\text{with}\quad
n_{0},\ldots,n_{N-1}\in\{0,\ldots,g-1\}
\ \text{and}\ 
n_{0},n_{N-1}\neq0,
\]
we have
\begin{align}
F_{N}(\alpha,\beta)
&=
\sum_{\substack{
1\le n_{0}<g\\
}}
\sum_{
0\le n_{1},\ldots,n_{N-2}<g
}
\sum_{1\le n_{N-1}<g}
e\biggl(
\alpha
\sum_{i=0}^{N-1}
n_{i}g^{i}
+
\beta
\sum_{i=0}^{N-1}
n_{N-i-1}g^{i}
\biggr)\\
&=
\sum_{\substack{
1\le n_{0}<g\\
}}
\sum_{
0\le n_{1},\ldots,n_{N-2}<g
}
\sum_{1\le n_{N-1}<g}
e\biggl(
\alpha
\sum_{i=0}^{N-1}
n_{i}g^{i}
+
\beta
\sum_{i=0}^{N-1}
n_{i}g^{N-i-1}
\biggr)\\
&=
\sum_{\substack{
1\le n_{0}<g\\
}}
\sum_{
0\le n_{1},\ldots,n_{N-2}<g
}
\sum_{1\le n_{N-1}<g}
\prod_{i=0}^{N-1}
e((\alpha g^{i} + \beta g^{N-i-1})n_{i})\\
&=
\biggl(
\sum_{1\le n_{0}<g}
e((\alpha + \beta g^{N-1})n_{0})
\biggr)
\times
\biggl(
\sum_{1\le n_{N-1}<g}
e((\alpha g^{N-1}+\beta)n_{N-1})
\biggr)\\
&\hspace{0.4\textwidth}
\times
\prod_{i=1}^{N-2}
\biggl(
\sum_{0\le n_{i}<g}
e((\alpha g^{i}+\beta g^{N-i-1})n_{i})
\biggr).
\end{align}
Bounding the sum over $n_{0},n_{N-1}$ trivially gives the inequality.
\end{proof}

\begin{proposition}
\label{prop:break_half}
For $M,N\in\mathbb{Z}$ with $3\le M\le N-1$, we have
\[
\Phi_{N}(\alpha,\beta)
=
\Phi_{M}(\alpha,g^{N-M}\beta)
\cdot
\Phi_{N-M+2}(g^{M-2}\alpha,\beta).
\]
\end{proposition}
\begin{proof}
We have
\begin{align}
\Phi_{N}(\alpha,\beta)
=
\prod_{i=1}^{N-2}
\varphi(\alpha g^{i} + \beta g^{N-i-1})
&=
\prod_{i=1}^{M-2}
\varphi(\alpha g^{i} + \beta g^{N-i-1})
\prod_{i=M-1}^{N-2}
\varphi(\alpha g^{i} + \beta g^{N-i-1}).
\end{align}
We then have
\begin{align}
\prod_{i=1}^{M-2}
\varphi(\alpha g^{i} + \beta g^{N-i-1})
&=
\prod_{i=1}^{M-2}
\varphi(\alpha g^{i} + g^{N-M}\beta g^{M-i-1})
=
\Phi_{M}(\alpha, g^{N-M}\beta)
\end{align}
and by changing the variable via $i\leadsto i+M-2$, we have
\begin{align}
\prod_{i=M-1}^{N-2}
\varphi(\alpha g^{i} + \beta g^{N-i-1})
=
\prod_{i=1}^{(N-M+2)-2}
\varphi(g^{M-2}\alpha g^{i} + \beta g^{(N-M+2)-i-1})
=
\Phi_{N-M+2}(g^{M-2}\alpha,\beta).
\end{align}
This completes the proof.
\end{proof}

\section{The \texorpdfstring{$L^{\infty}$}{L infinity}-bound}
\label{sec:L_infty}
In this section, we prove the pointwise or $L^{\infty}$-bound for $\Phi_{N}(\alpha,\beta)$
which will be used for the major arc estimate of $\rev{R}_{N}(a,q)$.
Some of the following results on exponential sums related to digital problems
are known but, for the ease of readers, we reprove most of them.
\begin{proposition}
\label{prop:phi_formula}
We have
\[
\varphi(\alpha)
=
e\biggl(\frac{g-1}{2}\alpha\biggr)
\frac{\sin\pi g\alpha}{\sin\pi\alpha}.
\]
\end{proposition}
\begin{proof}
This is an easy computation
with the sum formula for geometric progression.
\end{proof}

\begin{proposition}
\label{prop:strong_bound_phi}
For $\alpha\in\mathbb{R}$ and $u\in[0,g]$, we have
\[
\biggl|
\sum_{u\le n<g}
e(n\alpha)
\biggr|
\le
\min\biggl(g,\frac{1}{|\sin\pi\alpha|}\biggr)
\quad\text{and so}\quad
|\varphi(\alpha)|
\le
\min\biggl(g,\frac{1}{|\sin\pi\alpha|}\biggr).
\]
\end{proposition}
\begin{proof}
The bound $\le g$ is trivial.
For the other bound, we have
\[
\biggl|
\sum_{u\le n<g}
e(n\alpha)
\biggr|
=
\biggl|
\frac{e(g\alpha)-e(u\alpha)}{e(\alpha)-1}
\biggr|
\le
\frac{1}{|\sin\pi\alpha|}.
\]
Thus, we obtain the assertion.
\end{proof}

\begin{proposition}
\label{prop:strong_bound_phi_dev}
For $\alpha\in\mathbb{R}$ and $u\in[0,g]$, we have
\[
\biggl|
\sum_{u\le n<g}
ne(n\alpha)
\biggr|
\le
g
\min\biggl(g,\frac{1}{|\sin\pi\alpha|}\biggr),
\quad\text{and so}\quad
|\varphi'(\alpha)|
\le
2\pi g\min\biggl(g,\frac{1}{|\sin\pi\alpha|}\biggr).
\]
\end{proposition}
\begin{proof}
We have
\begin{align}
\biggl|
\sum_{u\le n<g}
ne(n\alpha)
\biggr|
=
\biggl|
\sum_{u\le n<g}
e(n\alpha)
\sum_{1\le m\le n}
1
\biggr|
=
\biggl|
\sum_{1\le m<g}
\sum_{\max(m,u)\le n<g}
e(n\alpha)
\biggr|
\le
g\max_{0\le v<g}
\biggl|
\sum_{v\le n<g}
e(n\alpha)
\biggr|.
\end{align}
Then the assertion follows from \cref{prop:strong_bound_phi}.
\end{proof}

\begin{proposition}[{\cite[Lemma~5.6]{TuxanidyPanario}}]
\label{prop:bound_phi}
For $\|\alpha\|\le\frac{1}{g}$, we have
\[
|\varphi(\alpha)|
\le
g
\exp
\biggl(
-\frac{\pi^{2}}{6}(g^{2}-1)\|\alpha\|^{2}
\biggr).
\]
\end{proposition}
\begin{proof}
For $\|\alpha\|=\frac{1}{g}$, \cref{prop:phi_formula} shows $|\varphi(\alpha)|=0$
and the assertion is trivial. Otherwise, by
\[
\sin\pi z
=
\pi z
\prod_{n\ge1}\biggl(1-\frac{z^{2}}{n^{2}}\biggr),\quad
\log(1-x)
=
-\sum_{k=1}^{\infty}\frac{x^{k}}{k}
\quad\text{for $|x|<1$}
\]
and \cref{prop:phi_formula}, since $\|\alpha\|<\frac{1}{g}$, we have
\begin{equation}
|\varphi(\alpha)|
=
\frac{\sin\pi g\|\alpha\|}{\sin\pi\|\alpha\|}
=
g
\prod_{n\ge1}
\biggl(1-\frac{g^{2}\|\alpha\|^{2}}{n^{2}}\biggr)
\biggl(1-\frac{\|\alpha\|^{2}}{n^{2}}\biggr)^{-1}
=
g\exp
\biggl(
-\sum_{k=1}^{\infty}\frac{\zeta(2k)}{k}(g^{2k}-1)\|\alpha\|^{2k}
\biggr).
\end{equation}
Since the terms of the above series are non-negative, the assertion follows.
\end{proof} 

\begin{proposition}[{\cite[Lemme~5]{MauduitRivat:carres}}]
\label{prop:phi_monotonic}
Let $\delta\in[0,\frac{2}{3g}]$.
Then, for $\|\alpha\|\ge\delta$, we have
$|\varphi(\alpha)|\le|\varphi(\delta)|$.
\end{proposition}
\begin{proof}
We first prove a preliminary estimate.
By the infinite product expansion of the sine function,
\[
\frac{\sin\pi t}{\sin\frac{2\pi t}{3}}
=
\frac{3}{2}\prod_{n\ge1}\frac{1-\frac{t^{2}}{n^{2}}}{1-\frac{4}{9}\frac{t^{2}}{n^{2}}}
=
\frac{3}{2}
\prod_{n\ge1}
\biggl(\frac{9}{4}-\frac{\frac{5}{4}}{1-\frac{4}{9}\frac{t^{2}}{n^{2}}}\biggr)
\]
so that the function
$t\mapsto\frac{\sin\pi t}{\sin\frac{2\pi t}{3}}$
is decreasing for $t\in[0,1]$. We thus have
\[
\biggl(\sin\frac{\pi}{g}\biggr)
\biggl|\varphi\biggl(\frac{2}{3g}\biggr)\biggr|
=
\biggl(\sin\frac{\pi}{g}\biggr)
\frac{\sin\frac{2\pi}{3}}{\sin\frac{2\pi}{3g}}
=
\biggl(\sin\frac{2\pi}{3}\biggr)
\frac{\sin\frac{\pi}{g}}{\sin\frac{2\pi}{3g}}
\ge
\biggl(\sin\frac{2\pi}{3}\biggr)
\frac{\sin\frac{\pi}{2}}{\sin\frac{\pi}{3}}
=1
\]
and so
\begin{equation}
\label{lem:phi_monotonic:pre_ineq}
\biggl|\varphi\biggl(\frac{2}{3g}\biggr)\biggr|
\ge
\biggl(\sin\frac{\pi}{g}\biggr)^{-1}.
\end{equation}
Also, for $0\le t\le\frac{1}{g}$, note that
\[
|\varphi(t)|
=
\frac{\sin\pi gt}{\sin\pi t}
=
g\prod_{n\ge1}\frac{1-\frac{g^{2}t^{2}}{n^{2}}}{1-\frac{t^{2}}{n^{2}}}
\]
is decreasing in $t$.
We now prove the assertion. When $\delta\le\|\alpha\|\le\frac{1}{g}$, we have
\[
|\varphi(\alpha)|
=
|\varphi(\|\alpha\|)|
\le
|\varphi(\delta)|
\]
by the monotonicity of $|\varphi(t)|$ for $t\in[0,\frac{1}{g}]$.
When $\frac{1}{g}\le\|\alpha\|\le\frac{1}{2}$, we have
\[
|\varphi(\alpha)|
=
|\varphi(\|\alpha\|)|
=
\frac{\sin\pi g\|\alpha\|}{\sin\pi\|\alpha\|}
\le
(\sin\pi\|\alpha\|)^{-1}
\le
\biggl(\sin\frac{\pi}{g}\biggr)^{-1}
\le
\biggl|\varphi\biggl(\frac{2}{3g}\biggr)\biggr|
\]
from the  inequality \cref{lem:phi_monotonic:pre_ineq}.
Since $0\le\delta\le\frac{2}{3g}\le\frac{1}{g}$,
we have
\[
|\varphi(\alpha)|
\le
\biggl|\varphi\biggl(\frac{2}{3g}\biggr)\biggr|
\le
|\varphi(\delta)|
\]
by the monotonicity of $|\varphi(t)|$ for $t\in[0,\frac{1}{g}]$.
This completes the proof.
\end{proof}

\begin{lemma}[{\cite[Lemma~4.8]{TuxanidyPanario}}]
\label{lem:consecutive_max}
For any $\alpha,\beta,\kappa,\lambda\in\mathbb{R}$, we have
\[
\max(
\|\alpha g^{\kappa}+\beta g^{\lambda}\|,
\|\alpha g^{\kappa+1}+\beta g^{\lambda-1}\|
)
\ge
\frac{\|\alpha(g^{2}-1)g^{\kappa}\|}{g+1}.
\]
\end{lemma}
\begin{proof}
We have
\[
(\alpha g^{\kappa}+\beta g^{\lambda})
-
g(\alpha g^{\kappa+1}+\beta g^{\lambda-1})
=
-\alpha(g^{2}-1)g^{\kappa}.
\]
By the triangle inequality for $\|\ast\|$, we  have
\begin{align}
\|\alpha(g^{2}-1)g^{\kappa}\|
&\le
\|\alpha g^{\kappa}+\beta g^{\lambda}\|
+
g\|\alpha g^{\kappa+1}+\beta g^{\lambda-1}\|\\
&\le
(g+1)\max(
\|\alpha g^{\kappa}+\beta g^{\lambda}\|,
\|\alpha g^{\kappa+1}+\beta g^{\lambda-1}\|
).
\end{align}
Thus the assertion follows.
\end{proof}

\begin{lemma}[{cf.\ \cite[Lemma~4.8]{TuxanidyPanario}}]
\label{lem:consecutive_max_phi}
For any $\alpha,\beta,\kappa,\lambda\in\mathbb{R}$, we have
\[
|\varphi(\alpha g^{\kappa}+\beta g^{\lambda})|
\cdot
|\varphi(\alpha g^{\kappa+1}+\beta g^{\lambda-1})|
\le
g^{2}\exp\biggl(-\frac{\pi^{2}}{6}\frac{g-1}{g+1}\|(g^{2}-1)g^{\kappa}\alpha\|^{2}\biggr).
\]
\end{lemma}
\begin{proof}
We have
\begin{equation}
\label{lem:consecutive_max_phi:first}
\begin{aligned}
&|\varphi(\alpha g^{\kappa}+\beta g^{\lambda})
\varphi(\alpha g^{\kappa+1}+\beta g^{\lambda-1})|\\
&=
|\varphi(
\min(
\|\alpha g^{\kappa}+\beta g^{\lambda}\|,
\|\alpha g^{\kappa+1}+\beta g^{\lambda-1}\|
))
\varphi(
\max(
\|\alpha g^{\kappa}+\beta g^{\lambda}\|,
\|\alpha g^{\kappa+1}+\beta g^{\lambda-1}\|
))|\\
&\le
g
|\varphi(
\max(
\|\alpha g^{\kappa}+\beta g^{\lambda}\|,
\|\alpha g^{\kappa+1}+\beta g^{\lambda-1}\|
))|.
\end{aligned}
\end{equation}
By \cref{lem:consecutive_max}, we have
\[
\max(
\|\alpha g^{\kappa}+\beta g^{\lambda}\|,
\|\alpha g^{\kappa+1}+\beta g^{\lambda-1}\|
)
\ge
\frac{\|(g^{2}-1)g^{\kappa}\alpha\|}{g+1}
\]
and also we have
\[
\frac{\|(g^{2}-1)g^{\kappa}\alpha\|}{g+1}
\le
\frac{1}{2g}
\le
\frac{2}{3g}.
\]
Thus, by using \cref{prop:phi_monotonic} with
\[
\delta
\coloneqq
\frac{\|(g^{2}-1)g^{\kappa}\alpha\|}{g+1}
\]
in \cref{lem:consecutive_max_phi:first}, we have
\begin{equation}
\label{lem:consecutive_max_phi:prefinal}
|\varphi(\alpha g^{\kappa}+\beta g^{\lambda})|
\cdot
|\varphi(\alpha g^{\kappa+1}+\beta g^{\lambda-1})|
\le
g\biggl|\varphi\biggl(\frac{\|(g^{2}-1)g^{\kappa}\alpha\|}{g+1}\biggr)\biggr|.
\end{equation}
Since
\[
0
\le
\frac{\|(g^{2}-1)g^{\kappa}\alpha\|}{g+1}
\le
\frac{1}{2},
\]
we have
\[
\biggl\|
\frac{\|(g^{2}-1)g^{\kappa}\alpha\|}{g+1}
\biggr\|
=
\frac{\|(g^{2}-1)g^{\kappa}\alpha\|}{g+1}
\le
\frac{1}{g}.
\]
Thus, by using \cref{prop:bound_phi} in \cref{lem:consecutive_max_phi:prefinal}, we arrive at
\[
|\varphi(\alpha g^{\kappa}+\beta g^{\lambda})|
\cdot
|\varphi(\alpha g^{\kappa+1}+\beta g^{\lambda-1})|
\le
g^{2}\exp\biggl(-\frac{\pi^{2}}{6}\frac{g-1}{g+1}\|(g^{2}-1)g^{\kappa}\alpha\|^{2}\biggr).
\]
This completes the proof.
\end{proof}

\begin{lemma}[{cf.\ \cite[Lemma~2.7]{DMRSS:ReversiblePrime}}]
\label{lem:geometric_progression_mod1}
For $\alpha\in\mathbb{R}\setminus\mathbb{Z}$, we have
\[
\|g^{i_{0}}\alpha\|\ge\frac{1}{g+1}
\quad\text{with}\quad
i_{0}
\coloneqq
\biggl[\frac{\log\frac{g}{(g+1)\|\alpha\|}}{\log g}\biggr].
\]
\end{lemma}
\begin{proof}
We have
\[
g^{i_{0}}\|\alpha\|
\ge
g^{\frac{\log\frac{g}{(g+1)\|\alpha\|}}{\log g}-1}\|\alpha\|
=
\frac{1}{g}\frac{g}{(g+1)\|\alpha\|}\cdot\|\alpha\|
=
\frac{1}{g+1}
\]
and
\[
g^{i_{0}}\|\alpha\|
\le
g^{\frac{\log\frac{g}{(g+1)\|\alpha\|}}{\log g}}\|\alpha\|
\le
\frac{g}{(g+1)\|\alpha\|}\cdot\|\alpha\|
=
1-\frac{1}{g+1}.
\]
The result now follows since $\|g^{i_{0}}\alpha\|=\|g^{i_{0}}\|\alpha\|\|$.
\end{proof}

\begin{lemma}
\label{lem:F_to_exp}
For $N\in\mathbb{Z}_{\ge4}$ and $\alpha,\beta\in\mathbb{R}$, we have
\[
|\Phi_{N}(\alpha,\beta)|
\le
g^{N-2}
\exp\biggl(
-\frac{\pi^{2}}{12}
\frac{g-1}{g+1}
\sum_{i=1}^{N-3}
\|(g^{2}-1)g^{i}\alpha\|^{2}
\biggr).
\]
\end{lemma}
\begin{proof}
By pairing the consecutive terms, we have
\begin{align}
|\Phi_{N}(\alpha,\beta)|
&=
\prod_{i=1}^{N-2}
|\varphi(
\alpha g^{i}
+
\beta g^{N-i-1}
)|^{\frac{1}{2}}
\prod_{i=1}^{N-2}
|\varphi(
\alpha g^{i}
+
\beta g^{N-i-1}
)|^{\frac{1}{2}}\\
&=
\prod_{i=1}^{N-2}
|\varphi(
\alpha g^{i}
+
\beta g^{N-i-1}
)|^{\frac{1}{2}}
\prod_{i=0}^{N-3}
|\varphi(
\alpha g^{i+1}
+
\beta g^{(N-i-1)-1}
)|^{\frac{1}{2}}.
\end{align}
By estimating the terms for $i=0,N-2$ trivially, we get
\begin{align}
|\Phi_{N}(\alpha,\beta)|
&\le
g
\prod_{i=1}^{N-3}
|\varphi(
\alpha g^{i}
+
\beta g^{N-i-1}
)|^{\frac{1}{2}}
|\varphi(
\alpha g^{i+1}
+
\beta g^{(N-i-1)-1}
)|^{\frac{1}{2}}.
\end{align}
By \cref{lem:consecutive_max_phi}, we obtain the assertion.
\end{proof}

\begin{lemma}[$L^{\infty}$-bound~{(cf.\ \cite[Lemma~2.8]{DMRSS:ReversiblePrime})}]
\label{lem:L_infty_bound_kd}
For $N\in\mathbb{Z}$, $\alpha\in\mathbb{R}$ and $k,\ell,d\in\mathbb{Z}$ with
\[
N\ge4,\quad
d\ge1,\quad
(k,d)=1,\quad
d\nmid (g^{2}-1)g^{N}k,
\]
we have
\[
\Phi_{N}\biggl(\alpha,\frac{k}{d}+\frac{\ell}{g^{3}-g}\biggr)
\ll
g^{N}
\exp\biggl(-c_{\infty}\cdot\frac{N}{\log d}\biggr)
\]
with some constant $c_{\infty}=c_{\infty}(g)\in(0,1)$,
where the implicit constant depends only on $g$.
\end{lemma}
\begin{proof}
Note that $d\nmid (g^{2}-1)g^{N}k$ implies $d\ge2$.
By \cref{prop:swap}, \cref{lem:F_to_exp} and
\[
\frac{\ell}{g^{3}-g}(g^{2}-1)g^{i}\in\mathbb{Z}
\quad\text{for $i\ge1$},
\]
we have
\begin{equation}
\label{lem:L_infty_bound_kd:F_to_exp}
\begin{aligned}
\biggl|\Phi_{N}\biggl(\alpha,\frac{k}{d}+\frac{\ell}{g^{3}-g}\biggr)\biggr|
&=
\biggl|\Phi_{N}\biggl(\frac{k}{d}+\frac{\ell}{g^{3}-g},\alpha\biggr)\biggr|\\
&\le
g^{N-2}
\exp\biggl(
-\frac{\pi^{2}}{12}
\frac{g-1}{g+1}
\sum_{i=1}^{N-3}
\biggl\|\frac{(g^{2}-1)g^{i}k}{d}\biggr\|^{2}
\biggr).
\end{aligned}
\end{equation}
Let
\[
J
\coloneqq
1+\biggl[\frac{\log\frac{gd}{g+1}}{\log g}\biggr]\ge1,\quad
L
\coloneqq
\biggl[\frac{N-3}{J}\biggr]
\and
\alpha_{\ell}
\coloneqq
\frac{(g^{2}-1)g^{\ell J+1}k}{d}.
\]
We then have
\begin{equation}
\label{lem:L_infty_bound_kd:inner_sum}
\sum_{i=1}^{N-3}
\biggl\|\frac{(g^{2}-1)g^{i}k}{d}\biggr\|^{2}
\ge
\sum_{0\le\ell<L}
\sum_{\ell J<i\le(\ell+1)J}
\biggl\|\frac{(g^{2}-1)g^{i}k}{d}\biggr\|^{2}
\ge
\sum_{0\le\ell<L}
\sum_{0\le i<J}
\|g^{i}\alpha_{\ell}\|^{2}.
\end{equation}
Since $d\nmid(g^{2}-1)g^{N}k$,
we have $\alpha_{\ell}\not\in\mathbb{Z}$ for all $\ell\in\{0,\ldots,L-1\}$
and so \cref{lem:geometric_progression_mod1} is applicable to the inner sum with $\alpha\coloneqq\alpha_{\ell}$.
Also, the same observation implies $\|\alpha_{\ell}\|\ge\frac{1}{d}$.
For $0\le\ell<L$, let
\[
i_{0,\ell}
\coloneqq
\biggl[\frac{\log\frac{g}{(g+1)\|\alpha_{\ell}\|}}{\log g}\biggr].
\]
We then have
\[
0\le i_{0,\ell}\le\biggl[\frac{\log\frac{gd}{(g+1)}}{\log g}\biggr]=J-1.
\]
Therefore, we can pick up the contribution of $i=i_{0,\ell}$
and use \cref{lem:geometric_progression_mod1} in \cref{lem:L_infty_bound_kd:inner_sum} to get
\begin{equation}
\label{lem:L_infty_bound_kd:inner_sum_lower_bound}
\sum_{i=1}^{N-3}
\biggl\|\frac{(g^{2}-1)g^{i}k}{d}\biggr\|^{2}
\ge
\sum_{0\le\ell<L}
\|g^{i_{0,\ell}}\alpha_{\ell}\|^{2}
\ge
\frac{L}{(g+1)^{2}}.
\end{equation}
By using the estimate
\[
L
\ge
\frac{N-3}{J}-1
\ge
\frac{1}{4}\frac{N}{J}-1
\ge
\frac{1}{4}
\frac{N}{1+\frac{\log d}{\log 2}}
-1
\ge
\frac{\log 2}{8}
\frac{N}{\log d}
-1
\]
in \cref{lem:L_infty_bound_kd:inner_sum_lower_bound}
and combining it with \cref{lem:L_infty_bound_kd:F_to_exp},
we obtain the result.
\end{proof}

\begin{lemma}[$L^{\infty}$-bound]
\label{lem:L_infty_bound}
For $N,\ell,q\in\mathbb{Z}$ with $N\ge 4$ and $q\ge1$, we have
\[
\frac{1}{q}
\sum_{\substack{
0\le k<q\\
q\nmid(g^{2}-1)g^{N}k
}}
\biggl|\Phi_{N}\biggl(\alpha,\frac{k}{q}+\frac{\ell}{g^{3}-g}\biggr)\biggr|
\ll
g^{N}
\exp\biggl(-c_{\infty}\cdot\frac{N}{\log q}\biggr)
\]
with some constant $c_{\infty}=c_{\infty}(g)\in(0,1)$,
where the implicit constant depends only on $g$.
\end{lemma}
\begin{proof}
By classifying the values of $(k,q)$, we have
\[
\frac{1}{q}
\sum_{\substack{
0\le k<q\\
q\nmid(g^{2}-1)g^{N}k
}}
\biggl|\Phi_{N}\biggl(\alpha,\frac{k}{q}+\frac{\ell}{g^{3}-g}\biggr)\biggr|
=
\frac{1}{q}
\sum_{d\mid q}
\sum_{\substack{
0\le k<q\\
q\nmid(g^{2}-1)g^{N}k\\
(k,q)=d
}}
\biggl|\Phi_{N}\biggl(\alpha,\frac{k}{q}+\frac{\ell}{g^{3}-g}\biggr)\biggr|.
\]
By changing variable via $d\leadsto\frac{q}{d}$ and $k\leadsto(\frac{q}{d})\cdot k$, we have
\[
\frac{1}{q}
\sum_{\substack{
0\le k<q\\
q\nmid(g^{2}-1)g^{N}k
}}
\biggl|\Phi_{N}\biggl(\alpha,\frac{k}{q}+\frac{\ell}{g^{3}-g}\biggr)\biggr|
=
\frac{1}{q}
\sum_{d\mid q}
\astsum_{\substack{
k\ \mod{d}\\
d\nmid(g^{2}-1)g^{N}k
}}
\biggl|\Phi_{N}\biggl(\alpha,\frac{k}{d}+\frac{\ell}{g^{3}-g}\biggr)\biggr|.
\]
By \cref{lem:L_infty_bound_kd}, we have
\[
\frac{1}{q}
\sum_{\substack{
0\le k<q\\
q\nmid(g^{2}-1)g^{N}k
}}
\biggl|\Phi_{N}\biggl(\alpha,\frac{k}{q}+\frac{\ell}{g^{3}-g}\biggr)\biggr|
\ll
g^{N}\exp\biggl(-c_{\infty}\cdot\frac{N}{\log q}\biggr)
\frac{1}{q}
\sum_{d\mid q}
\phi(d)
\ll
g^{N}\exp\biggl(-c_{\infty}\cdot\frac{N}{\log q}\biggr).
\]
This completes the proof.
\end{proof}

\section{The \texorpdfstring{$L^{1}$}{L1}-bound}
\label{sec:L1_bound}
In this section we  prove several $L^{1}$-bounds for $\Phi_{N}(\alpha,\beta)$,
which will be used as ingredients for the minor arc estimate of $\rev{R}_{N}(a,q)$.
We first consider the $L^{1}$-moment
taken over fractions having some power of $g$ as the denominator.
We shall use the constant
\[
C_{g}
\coloneqq
\frac{2}{\pi}\log\cot\frac{\pi}{2g}
+\frac{1}{g\sin\frac{\pi}{2g}}
+1
\]
defined in \cref{def:C_g}. Note that we have
\begin{equation}
\label{Cg:lower_bound}
C_{g}
>
1
\end{equation}
since $\frac{\pi}{2g}\in(0,\frac{\pi}{4}]$.

\begin{lemma}[{cf. \cite[Lemme~6]{MauduitRivat:Gelfond}}]
\label{lem:basic_sum}
For $g\in\mathbb{Z}_{\ge2}$ and $\theta\in\mathbb{R}$, we have
\[
\sum_{0\le h<g}\min\biggl(g,\frac{1}{|\sin\pi(\frac{h}{g}+\theta)|}\biggr)
\le
C_{g}g.
\]
\end{lemma}
\begin{proof}
This follows from Lemme~6 of \cite{MauduitRivat:Gelfond}
but, for the ease of readers, we give a complete proof 
here in the form that is enough for us.
Let us write $S$ for the left-hand side of the assertion.
Since we can think of $S$ as a sum over the residues $h\ \mod{g}$,
we can shift $h$ to shift $\theta$ to a real number $\frac{\delta}{g}$ with $|\delta|\le\frac{1}{2}$.
By changing $h\ \mod{g}$ by $-h\ \mod{g}$ if necessary,
we can further assume $\delta\in[0,\frac{1}{2}]$.
By bounding the term with $h=0$ by $g$
and by noting that $0<\frac{h+\delta}{g}<1$ for $1\le h<g$,
we have
\[
S
\le
\sum_{1\le h\le g-1}\frac{1}{\sin\pi(\frac{h+\delta}{g})}
+
g
\le
\sum_{1\le h\le g-2}\frac{1}{\sin\pi(\frac{h+\delta}{g})}
+
\frac{1}{\sin\pi(\frac{1-\delta}{g})}
+
g.
\]
Since the function $x\mapsto\frac{1}{\sin\pi x}$ is convex downwards for $0<x<1$, we have
\[
\frac{1}{\sin\pi(\frac{h+\delta}{g})}
\le
\int_{h-\frac{1}{2}}^{h+\frac{1}{2}}
\frac{du}{\sin\pi(\frac{u+\delta}{g})}
\quad\text{for}\quad
1\le h\le g-2
\]
and so
\[
S
\le
\int_{\frac{1}{2}}^{g-\frac{3}{2}}
\frac{du}{\sin\pi(\frac{u+\delta}{g})}
+
\frac{1}{\sin\pi(\frac{1-\delta}{g})}
+
g
\eqqcolon
S(\delta).
\]
By the convexity of the function $x\mapsto\frac{1}{\sin\pi x}$,
the function $S(\delta)$ is convex downwards with respect to $\delta\in[0,\frac{1}{2}]$.
Thus, its maximum of $S(\delta)$ for $\delta\in[0,\frac{1}{2}]$
is taken at either $\delta=0$ or $\delta=\frac{1}{2}$.
We have
\begin{align}
S(\tfrac{1}{2})-S(0)
&=
\frac{1}{\sin\frac{\pi}{2g}}
-
\frac{1}{\sin\frac{\pi}{g}}
+
\int_{g-\frac{3}{2}}^{g-1}
\frac{du}{\sin\frac{\pi u}{g}}
-
\int_{\frac{1}{2}}^{1}
\frac{du}{\sin\frac{\pi u}{g}}\\
&\ge
\frac{1}{\sin\frac{\pi}{2g}}
-
\frac{1}{\sin\frac{\pi}{g}}
+
\frac{1}{2\sin\frac{3\pi}{2g}}
-
\frac{1}{2\sin\frac{\pi}{2g}}\\
&=
\frac{1}{2\sin\frac{\pi}{2g}}
+
\frac{1}{2\sin\frac{3\pi}{2g}}
-
\frac{1}{\sin\frac{\pi}{g}}
\ge0,
\end{align}
where once again we used the convexity of $x\mapsto\frac{1}{\sin\pi x}$ in the last inequality.
This shows
\[
S
\le
S(\tfrac{1}{2})
=
\int_{1}^{g-1}
\frac{du}{\sin\frac{\pi u}{g}}
+
\frac{1}{\sin\frac{\pi}{2g}}
+
g
=
\frac{2}{\pi}g\log\cot\frac{\pi}{2g}
+
\frac{1}{\sin\frac{\pi}{2g}}
+
g
=
C_{g}g
\]
where we have used $(\log\tan\frac{x}{2})'=\frac{1}{\sin x}$.
This completes the proof.
\end{proof}

The next lemma is comparable to Lemma~5.1 of \cite{Maynard:PnP}.
Our sum is simpler in the sense that we have complete exponential sums $\mod{g}$
though we have to insert a new extra shift $\theta_{i}$.
We also avoid taking the maximum over the less significant digits.
\begin{lemma}
\label{lem:prelim_L1}
For $M\in\mathbb{Z}_{\ge2}$ and a sequence of real numbers $(\theta_{i})_{i=1}^{M-2}$, we have
\[
\sum_{0\le h<g^{M}}
\prod_{i=1}^{M-2}
\min\biggl(g,\frac{1}{|\sin\pi(hg^{-(i+1)}+\theta_{i})|}\biggr)
\le
(C_{g}g)^{M}.
\]
\end{lemma}
\begin{proof}
Let $S$ be the left-hand side of the inequality.
By letting
\[
h=ng^{M-1}+h_{-}
\quad\text{with}\quad
n\in\{0,\ldots,g-1\}
\quad\text{and}\quad
h_{-}\in[0,g^{M-1}),
\]
we can  rewrite $S$ as
\[
S
=
\sum_{0\le h_{-}<g^{M-1}}
\sum_{0\le n<g}
\prod_{i=1}^{M-2}
\min\biggl(g,\frac{1}{|\sin\pi(ng^{M-i-2}+h_{-}g^{-(i+1)}+\theta_{i})|}\biggr).
\]
However, in the above product, $ng^{M-i-2}$ is always an integer and so
\begin{align}
S
&=
\sum_{0\le h_{-}<g^{M-1}}
\sum_{0\le n<g}
\prod_{i=1}^{M-2}
\min\biggl(g,\frac{1}{|\sin\pi(h_{-}g^{-(i+1)}+\theta_{i})|}\biggr)\\
&=
g
\sum_{0\le h_{-}<g^{M-1}}
\prod_{i=1}^{M-2}
\min\biggl(g,\frac{1}{|\sin\pi(h_{-}g^{-(i+1)}+\theta_{i})|}\biggr).
\end{align}
Since $C_{g}>1$ by \cref{Cg:lower_bound}, it thus suffices to show
\begin{equation}
\label{lem:prelim_L1:goal}
S_{M}
\coloneqq
\sum_{0\le h<g^{M}}
\prod_{i=1}^{M-1}
\min\biggl(g,\frac{1}{|\sin\pi(hg^{-(i+1)}+\theta_{i})|}\biggr)
\le
(C_{g}g)^{M}
\end{equation}
for $M\in\mathbb{Z}_{\ge1}$ and $(\theta_{i})_{i=1}^{M-1}$. We use induction on $M\in\mathbb{Z}_{\ge1}$.

We first consider the initial case $M=1$. In this case, we have
\[
S_{1}
=
\sum_{0\le h<g}
1
=
g
\le
C_{g}g
\]
since $C_{g}>1$ by \cref{Cg:lower_bound}.
This proves \cref{lem:prelim_L1:goal} for the initial case $M=1$.

We next assume the $M$-th case of \cref{lem:prelim_L1:goal} with $M\ge1$
and show the $(M+1)$-th case of \cref{lem:prelim_L1:goal}.
Let us express the summation variable $h\in[0,g^{M+1})$ as
\begin{equation}
\label{lem:prelim_L1:h_expansion}
h=ng^{M}+h_{-}
\quad\text{with}\quad
n\in\{0,\ldots,g-1\}
\ \text{and}\ 
h_{-}\in[0,g^{M})\cap\mathbb{Z}.
\end{equation}
We then have
\[
S_{M+1}
=
\sum_{0\le h_{-}<g^{M}}
\sum_{0\le n<g}
\prod_{i=1}^{M}
\min\biggl(g,\frac{1}{|\sin\pi(ng^{M-(i+1)}+h_{-}g^{-(i+1)}+\theta_{i})|}\biggr).
\]
In the above product, we have $ng^{M-(i+1)}\in\mathbb{Z}$ for $i=1,\ldots,M-1$. This gives
\begin{align}
S_{M+1}
&=
\sum_{0\le h_{-}<g^{M}}
\prod_{i=1}^{M-1}
\min\biggl(g,\frac{1}{|\sin\pi(hg^{-(i+1)}+\theta_{i})|}\biggr)\\
&\hspace{0.1\textwidth}
\times
\sum_{0\le n<g}
\min\biggl(g,\frac{1}{|\sin\pi(ng^{-1}+h_{-}g^{-(M+1)}+\theta_{M})|}\biggr).
\end{align}
We can then apply \cref{lem:basic_sum} with $\theta=h_{-}g^{-(M+1)}+\theta_{M}$ to the inner sum.
This gives
\[
S_{M+1}
\le
C_{g}g
\sum_{0\le h_{-}<g^{M}}
\prod_{i=1}^{M-1}
\min\biggl(g,\frac{1}{|\sin\pi(h_{-}g^{-(i+1)}+\theta_{i})|}\biggr)
\le
(C_{g}g)^{M+1}
\]
by the induction hypothesis.
This completes the proof.
\end{proof}

\begin{lemma}[{Discrete $L^{1}$-bound}]
\label{lem:L1_discrete}
For $M\in\mathbb{Z}_{\ge3}$, $\theta,\beta\in\mathbb{R}$, we have
\[
\sum_{0\le h<g^{M}}
\biggl|
\Phi_{M}\biggl(\frac{h}{g^{M}}+\theta,\beta\biggr)
\biggr|
\le
(C_{g}g)^{M}.
\]
\end{lemma}
\begin{proof}
By \cref{prop:swap}, we have
\begin{align}
\sum_{0\le h<g^{M}}
\biggl|
\Phi_{M}\biggl(\frac{h}{g^{M}}+\theta,\beta\biggr)
\biggr|
&=
\sum_{0\le h<g^{M}}
\biggl|
\Phi_{M}\biggl(\beta,\frac{h}{g^{M}}+\theta\biggr)
\biggr|\\
&\le
\sum_{0\le h<g^{M}}
\prod_{i=1}^{M-2}
\biggl|\varphi\biggl(
\biggl(\frac{h}{g^{M}}+\theta\biggr)g^{M-i-1}
+
\beta g^{i}
\biggr)\biggr|\\
&=
\sum_{0\le h<g^{M}}
\prod_{i=1}^{M-2}
|\varphi(hg^{-(i+1)}+\theta_{i})|,
\end{align}
where $\theta_{i}\coloneqq\theta g^{M-i-1}+\beta g^{i}$.
By using \cref{prop:strong_bound_phi} and \cref{lem:prelim_L1}, we have
\begin{align}
\sum_{0\le h<g^{M}}
\biggl|
\Phi_{M}\biggl(\frac{h}{g^{M}}+\theta,\beta\biggr)
\biggr|
&\le
\sum_{0\le h<g^{M}}
\prod_{i=1}^{M-2}
\min\biggl(g,\frac{1}{|\sin\pi(hg^{-(i+1)}+\theta_{i})|}\biggr)
\le
(C_{g}g)^{M}.
\end{align}
This completes the proof.
\end{proof}

\begin{lemma}[{Continuous $L^{1}$-bound}]
\label{lem:L1_continuous}
For $M\in\mathbb{Z}_{\ge3}$ and $\beta\in\mathbb{R}$, we have
\[
\int_{0}^{1}
|\Phi_{M}(\alpha,\beta)|
d\alpha
\le
C_{g}^{M}.
\]
\end{lemma}
\begin{proof}
We decompose the interval into $g^{M}$ parts to obtain
\begin{align}
\int_{0}^{1}
|\Phi_{M}(\alpha,\beta)|
d\alpha
&=
\sum_{0\le h<g^{M}}
\int_{0}^{\frac{1}{g^{M}}}
\biggl|\Phi_{M}\biggl(\frac{h}{g^{M}}+\theta,\beta\biggr)\biggr|
d\theta\\
&=
\int_{0}^{\frac{1}{g^{M}}}
\sum_{0\le h<g^{M}}
\biggl|\Phi_{M}\biggl(\frac{h}{g^{M}}+\theta,\beta\biggr)\biggr|
d\theta.
\end{align}
By \cref{lem:L1_discrete}, we obtain
\[
\int_{0}^{1}
|\Phi_{M}(\alpha,\beta)|
d\alpha
\le
(C_{g}g)^{M}
\int_{0}^{\frac{1}{g^{M}}}
d\theta
=
C_{g}^{M}.
\]
This completes the proof.
\end{proof}

\begin{lemma}[{Discrete $L^{1}$-bound for derivative}]
\label{lem:L1_discrete_dev}
For $M\in\mathbb{Z}_{\ge3}$, $\beta\in\mathbb{R}$, we have
\[
\sum_{0\le h<g^{M}}
\biggl|
\frac{\partial \Phi_{M}}{\partial\alpha}\biggl(\frac{h}{g^{M}}+\theta,\beta\biggr)
\biggr|
\le
2\pi g^{M}(C_{g}g)^{M}.
\]
\end{lemma}
\begin{proof}
Recall the definition
\[
\Phi_{M}(\alpha,\beta)
=
\prod_{i=1}^{M-2}
\varphi(\alpha g^{i}+\beta g^{M-i-1}).
\]
By taking its derivative with respect to $\alpha$, we have
\begin{align}
\frac{\partial\Phi_{M}}{\partial\alpha}(\alpha,\beta)
&=
\sum_{i=1}^{M-2}
g^{i}\varphi'(\alpha g^{i}+\beta g^{M-i-1})
\prod_{\substack{j=1\\i\neq j}}^{M-2}
\varphi(\alpha g^{j}+\beta g^{M-j-1}).
\end{align}
By using \cref{prop:strong_bound_phi} and \cref{prop:strong_bound_phi_dev}, we have
\begin{align}
\biggl|\frac{\partial\Phi_{M}}{\partial\alpha}(\alpha,\beta)\biggr|
&\le
2\pi
\sum_{i=0}^{M-2}
g^{i+1}
\prod_{j=1}^{M-2}
\min\biggl(g,\frac{1}{|\sin\pi(\alpha g^{j}+\beta g^{M-j-1})|}\biggr)\\
&\le
2\pi g^{M}
\prod_{i=1}^{M-2}
\min\biggl(g,\frac{1}{|\sin\pi(\alpha g^{i}+\beta g^{M-i-1})|}\biggr)\\
&=
2\pi g^{M}
\prod_{i=1}^{M-2}
\min\biggl(g,\frac{1}{|\sin\pi(\alpha g^{M-i-1}+\beta g^{i})|}\biggr).
\end{align}
By using this estimate, we have
\begin{align}
\sum_{0\le h<g^{M}}
\biggl|
\frac{\partial\Phi_{M}}{\partial\alpha}\biggl(\frac{h}{g^{M}}+\theta,\beta\biggr)
\biggr|
&\le
2\pi 
g^{M}
\sum_{0\le h<g^{M}}
\prod_{i=1}^{M-2}
\min\biggl(g,\frac{1}{|\sin\pi(\frac{h}{g^{M}}+\theta)g^{M-i-1}+\beta g^{i})|}\biggr)\\
&=
2\pi
g^{M}
\sum_{0\le h<g^{M}}
\prod_{i=1}^{M-2}
\min\biggl(g,\frac{1}{|\sin\pi(hg^{-(i+1)}+\theta_{i})|}\biggr)
\end{align}
with $\theta_{i}\coloneqq\theta g^{M-i-1}+\beta g^{i}$.
Then the assertion follows from \cref{lem:prelim_L1}.
\end{proof}

\begin{lemma}[{Continuous $L^{1}$-bound for derivative}]
\label{lem:L1_continuous_dev}
For $M\in\mathbb{Z}_{\ge3}$ and $\beta\in\mathbb{R}$, we have
\[
\int_{0}^{1}
\biggl|\frac{\partial\Phi_{M}}{\partial\alpha}(\alpha,\beta)\biggr|
d\alpha
\le
2\pi g^{M}C_{g}^{M}.
\]
\end{lemma}
\begin{proof}
We decompose the interval into $g^{M}$ parts to obtain
\begin{align}
\int_{0}^{1}
\biggl|\frac{d\Phi_{M}}{d\alpha}(\alpha,\beta)\biggr|
d\alpha
&=
\sum_{0\le h<g^{M}}
\int_{0}^{\frac{1}{g^{M}}}
\biggl|\frac{\partial\Phi_{M}}{\partial\alpha}\biggl(\frac{h}{g^{M}}+\theta,\beta\biggr)\biggr|
d\theta\\
&=
\int_{0}^{\frac{1}{g^{M}}}
\sum_{0\le h<g^{M}}
\biggl|\frac{\partial\Phi_{M}}{\partial\alpha}\biggl(\frac{h}{g^{M}}+\theta,\beta\biggr)\biggr|
d\theta.
\end{align}
By \cref{lem:L1_discrete_dev}, we obtain
\[
\int_{0}^{1}
\biggl|\frac{\partial\Phi_{M}}{\partial\alpha}(\alpha,\beta)\biggr|
d\alpha
\le
2\pi
g^{M}(C_{g}g)^{M}
\int_{0}^{\frac{1}{g^{M}}}
d\theta
=
2\pi g^{M}
C_{g}^{M}.
\]
This completes the proof.
\end{proof}

\section{Large sieve estimates}
\label{sec:large_sieve}
In this section we consider the $L^{1}$-moment taken over Farey fractions.
Our main tool, as in the estimation of various large sieves,
is the Gallagher--Sobolev inequality.

\begin{lemma}[Galllagher--Sobolev inequality]
\label{lem:Gallagher_Sobolev}
Let $f\colon[0,1]\to\mathbb{C}$ be a function of class $C^{1}$ of period $1$,
$\delta>0$ and $(\alpha_{i})_{i=1}^{R}$ be a sequence of real numbers
which is $\delta$-spaced $\mod{1}$, i.e.\ 
\[
\|\alpha_{i}-\alpha_{j}\|\ge\delta
\quad\text{for any}\quad
i,j\in\{1,\ldots,R\}
\ \text{with}\ 
i\neq j.
\]
We then have
\[
\sum_{i=1}^{R}|f(\alpha_{i})|
\le
\frac{1}{\delta}\int_{0}^{1}|f(\alpha)|d\alpha
+
\frac{1}{2}
\int_{0}^{1}|f'(\alpha)|d\alpha.
\]
\end{lemma}
\begin{proof}
See Lemma~1.2 of \cite[p.~2]{Montgomery:Topics}.
\end{proof}

\begin{lemma}[Large sieve estimate]
\label{lem:large_sieve_estimate}
For $M\in\mathbb{N}_{\ge3}$, $R,\theta,\beta\in\mathbb{R}$ with $R\ge2$,
we have
\[
\sum_{r\le R}
\astsum_{b\ \mod{r}}
\max_{|\eta|\le\frac{1}{4}R^{-2}}
\biggl|
\Phi_{M}\biggl(\frac{b}{r}+\theta+\eta,\beta\biggr)
\biggr|
\ll
(g^{M}+R^{2})C_{g}^{M}.
\]
\end{lemma}
\begin{proof}
For $b,r$ appearing in the left-hand side,
by the continuity in $\eta$,
we can take $\eta_{b,r}$ such that
\begin{equation}
\label{lem:large_sieve_estimate:eta_aq}
\max_{|\eta|\le\frac{1}{4}R^{-2}}
\biggl|
\Phi_{M}\biggl(\frac{b}{r}+\theta+\eta,\beta\biggr)
\biggr|
=
\biggl|
\Phi_{M}\biggl(\frac{b}{r}+\theta+\eta_{b,r},\beta\biggr)
\biggr|
\and
|\eta_{b,r}|\le\tfrac{1}{4}R^{-2}.
\end{equation}
For any two distinct $(b,r)$ and $(b',r')$ in the above sum and $m\in\mathbb{Z}$,
\begin{align}
\biggl|
\biggl(\frac{b}{r}+\theta+\eta_{b,r}\biggr)
-
\biggl(\frac{b'}{r'}+\theta+\eta_{b',r'}\biggr)
-m
\biggr|
&\ge
\biggl|\frac{b}{r}-\frac{b'}{r'}-m\biggr|
-|\eta_{b,r}|-|\eta_{b',r'}|\\
&
\ge
\frac{1}{rr'}-\frac{1}{2R^{2}}
\ge
\frac{1}{2R^{2}}
\end{align}
and so the real numbers
\[
\frac{b}{r}+\theta+\eta_{b,r}
\]
appearing in the sum are $\frac{1}{2}R^{-2}$-spaced $\mod{1}$.
Thus, \cref{lem:Gallagher_Sobolev} and \cref{lem:large_sieve_estimate:eta_aq} give
\begin{align}
\sum_{r\le R}
\astsum_{b\ \mod{r}}
\max_{|\eta|\le\frac{1}{4}R^{-2}}
\biggl|
\Phi_{M}\biggl(\frac{b}{r}+\theta+\eta,\beta\biggr)
\biggr|
&=
\sum_{r\le R}
\astsum_{b\ \mod{r}}
\biggl|
\Phi_{M}\biggl(\frac{b}{r}+\theta+\eta_{b,r},\beta\biggr)
\biggr|\\
&\ll
R^{2}\int_{0}^{1}|\Phi_{M}(\alpha,\beta)|d\alpha
+
\int_{0}^{1}\biggl|\frac{\partial\Phi_{M}}{\partial\alpha}(\alpha,\beta)\biggr|d\alpha.
\end{align}
Then the assertion follows from \cref{lem:L1_continuous} and \cref{lem:L1_continuous_dev}.
\end{proof}

\section{Hybrid bound}
\label{sec:hybrid_bound}
We now combine the results of \cref{sec:L1_bound} and \cref{sec:large_sieve}
to obtain a hybrid $L^{1}$-bound.
\begin{lemma}[Hybrid bound]
\label{lem:hybrid}
For $N\in\mathbb{Z}_{\ge8}$, $R\ge1$, $H\ge 4g^{8}$ with
\[
R^{2}H\le 4g^{N}
\]
and $\beta\in\mathbb{R}$, we have
\[
\sum_{r\le R}\ 
\astsum_{b\ \mod{r}}
\sum_{\substack{
g^{N}|\eta|\le H\\
\frac{g^{N}b}{r}+g^{N}\eta\in\mathbb{Z}
}}
\biggl|\Phi_{N}\biggl(\frac{b}{r}+\eta,\beta\biggr)\biggr|
\ll
g^{N}
(R^{2}H)^{\alpha_{g}},
\]
where the exponent $\alpha_{g}$ is given by
\[
\alpha_{g}
\coloneqq
\frac{\log C_{g}}{\log g}
=
\frac{\log(\frac{2}{\pi}\log\cot\frac{\pi}{2g}+\frac{1}{g\sin\frac{\pi}{2g}}+1)}{\log g}
\]
and the implicit constant depends only on $g$.
\end{lemma}
\begin{proof}
Let $S$ be the left-hand side of the assertion.
We take $L,M\in\mathbb{Z}$ chosen later such that
\begin{equation}
\label{lem:hybrid:LM_range}
3\le L \le N-M+1
\and
3\le M \le N-1.
\end{equation}
We then decompose $\Phi_{N}$ by using \cref{prop:break_half} as
\begin{align}
&\Phi_{N}\biggl(\frac{b}{r}+\eta,\beta\biggr)\\
&=
\Phi_{N-M+2}\biggl(\frac{b}{r}+\eta,g^{M-2}\beta\biggr)
\Phi_{M}\biggl(g^{N-M}\biggl(\frac{b}{r}+\eta\biggr),\beta\biggr)\\
&=
\Phi_{L}\biggl(\frac{b}{r}+\eta,g^{N-L}\beta\biggr)
\Phi_{N-(L+M)+4}\biggl(g^{L-2}\biggl(\frac{b}{r}+\eta\biggr),g^{M-2}\beta\biggr)
\Phi_{M}\biggl(g^{N-M}\biggl(\frac{b}{r}+\eta\biggr),\beta\biggr).
\end{align}
By using the trivial bound
\[
\Phi_{N-(L+M)+4}\biggl(g^{L-2}\biggl(\frac{b}{r}+\eta\biggr),g^{M-2}\beta\biggr)
\ll
g^{N-(L+M)}
\]
and writing $\widetilde{\beta}\coloneqq g^{N-L}\beta$, we have
\begin{equation}
\label{lem:hybrid:after_Phi0_trivial}
S
\ll
g^{N-(L+M)}
\sum_{r\le R}
\astsum_{b\ \mod{r}}
\sum_{\substack{
g^{N}|\eta|\le H\\
\frac{g^{N}b}{r}+g^{N}\eta\in\mathbb{Z}
}}
\biggl|\Phi_{L}\biggl(\frac{b}{r}+\eta,\widetilde{\beta}\biggr)\biggr|
\biggl|\Phi_{M}\biggl(g^{N-M}\biggl(\frac{b}{r}+\eta\biggr),\beta\biggr)\biggr|.
\end{equation}
In this sum, by the assumption $R^{2}H\le 4g^{N}$,
we have
\[
|\eta|
\le
Hg^{-N}
\le
4R^{-2}.
\]
We can thus estimate \cref{lem:hybrid:after_Phi0_trivial} as
\begin{equation}
\label{lem:hybrid:after_Phi12_decomp}
S
\ll
g^{N-(L+M)}
\sum_{r\le R}
\astsum_{b\ \mod{r}}
\max_{|\epsilon|\le 4R^{-2}}
\biggl|\Phi_{L}\biggl(\frac{b}{r}+\epsilon,\widetilde{\beta}\biggr)\biggr|
S(b,r),
\end{equation}
where
\[
S(b,r)
\coloneqq
\sum_{\substack{
g^{N}|\eta|\le H\\
\frac{g^{N}b}{r}+g^{N}\eta\in\mathbb{Z}
}}
\biggl|\Phi_{M}\biggl(g^{N-M}\biggl(\frac{b}{r}+\eta\biggr),\beta\biggr)\biggr|.
\]
In the sum $S(b,r)$, we change variable from $\eta$ to $h\coloneqq\frac{g^{N}b}{r}+g^{N}\eta\in\mathbb{Z}$.
This gives
\begin{equation}
\label{lem:hybrid:inner_change_variable}
S(b,r)
=
\sum_{|h-\frac{g^{N}b}{r}|\le H}
\biggl|\Phi_{M}\biggl(\frac{h}{g^{M}},\beta\biggr)\biggr|.
\end{equation}
We then partition the sum over $h$ into the sums over subintervals of length $g^{M}$
and apply the $L^{1}$ bound (\cref{lem:L1_discrete}).
Since there are $\ll Hg^{-M}+1$ such subintervals,
we then obtain
\[
S(b,r)
\ll
(1+Hg^{-M})
\sum_{0\le h<g^{M}}
\biggl|\Phi_{M}\biggl(\frac{h}{g^{M}},\beta\biggr)\biggr|
\le
(1+Hg^{-M})
(C_{g}g)^{M}.
\]
On inserting this bound into \cref{lem:hybrid:after_Phi12_decomp}, we get
\begin{equation}
\label{lem:hybrid:after_inner_sum}
S
\ll
g^{N-L}
(1+Hg^{-M})
C_{g}^{M}
\sum_{r\le R}
\astsum_{b\ \mod{r}}
\max_{|\epsilon|\le 4R^{-2}}
\biggl|\Phi_{L}\biggl(\frac{b}{r}+\epsilon,\widetilde{\beta}\biggr)\biggr|.
\end{equation}
In \cref{lem:hybrid:after_inner_sum},
for each $(b,r)$, we can take a real number $\epsilon_{b,r}$ with
\[
\sup_{|\epsilon|\le R^{-2}}
\biggl|\Phi_{L}\biggl(\frac{b}{r}+\epsilon,\widetilde{\beta}\biggr)\biggr|
=
\biggl|\Phi_{L}\biggl(\frac{b}{r}+\epsilon_{b,r},\widetilde{\beta}\biggr)\biggr|
\quad\text{with}\quad
|\epsilon_{b,r}|\le 4R^{-2}.
\]
By classifying the terms according to which one of the intervals
\[
\biggl[\frac{i}{4}R^{-2},\frac{i+1}{4}R^{-2}\biggr]
\quad\text{for $i\in\{-16,-15,\ldots,+15,+15\}$}
\]
contains $\epsilon_{b,r}$, \cref{lem:hybrid:after_inner_sum} can be bounded as
\begin{equation}
\label{lem:hybrid:after_classify_epsilon}
S
\ll
g^{N-L}
(1+Hg^{-M})
C_{g}^{M}
\max_{|i|\le 16}
\sum_{r\le R}
\astsum_{b\ \mod{r}}
\max_{|\epsilon|\le\frac{1}{4}R^{-2}}
\biggl|\Phi_{L}\biggl(\frac{b}{r}+\frac{i}{4}R^{-2}+\epsilon,\widetilde{\beta}\biggr)\biggr|.
\end{equation}
By \cref{lem:large_sieve_estimate}, we arrive at
\begin{equation}
\label{lem:hybrid:before_optimization}
\begin{aligned}
S
&\ll
g^{N-L}(1+Hg^{-M})(g^{L}+R^{2})C_{g}^{M+L}\\
&\ll
g^{N}
\times
\max(1,R^{2}g^{-L})C_{g}^{L}
\times
\max(1,Hg^{-M})C_{g}^{M}.
\end{aligned}
\end{equation}
When
\begin{equation}
\label{lem:hybrid:optimization_cond}
\frac{C_{g}}{g}<1
\end{equation}
does not hold, the bound \cref{lem:hybrid:before_optimization} is weaker than the trivial bound
\[
S
\ll
g^{N}R^{2}H.
\]
We may thus assume \cref{lem:hybrid:optimization_cond} to optimize our $L,M$.
By $C_{g}>1$ and \cref{lem:hybrid:optimization_cond}, the general quantity
\begin{equation}
\label{lem:hybrid:general_optimization}
\max(1,Xg^{-K})C_{g}^{K}
\quad\text{with}\quad
X\ge1\ \text{and}\ K\in\mathbb{N}
\end{equation}
is decreasing for $g^{K}\le X$ and increasing for $g^{K}\ge X$.
Thus, \cref{lem:hybrid:general_optimization}
is minimized when $g^{K}\asymp X$.
Therefore, to minimize \cref{lem:hybrid:before_optimization},
we try to take $L,M$ so that
\[
g^{L}\asymp R^{2}
\and
g^{M}\asymp H.
\]
By the assumptions $R\ge1$, $H\ge 4g^{8}$ and $R^{2}H\le 4g^{N}$,
we can take $M\in\mathbb{Z}_{\ge6}$ with \cref{lem:hybrid:LM_range} such that
\begin{equation}
\label{lem:hybrid:M_choice}
4g^{M+2}\le H<4g^{M+3}.
\end{equation}
We shall then take $L\in\mathbb{N}$ by
\begin{equation}
\label{lem:hybrid:L_choice}
g^{L-3}\le R^{2}<g^{L-2}.
\end{equation}
Combined with \cref{lem:hybrid:M_choice},
this choice is in the range \cref{lem:hybrid:LM_range} since
\[
R\ge1
\ \text{and}\ 
4g^{L+M-1}
\le R^{2}H
\le 4g^{N}
\quad\text{so that}\quad
3\le L\le N-M+1
\ \text{and}\ 
6\le M\le N-2
\]
by the assumptions $R\ge1$ and $R^{2}H\le 4g^{N}$.
By \cref{lem:hybrid:before_optimization}, we then have
\begin{equation}
\label{lem:hybrid:prefinal}
S
\ll
g^{N}C_{g}^{M+L}.
\end{equation}
By \cref{lem:hybrid:M_choice} and \cref{lem:hybrid:L_choice},
we have
\[
M=\frac{\log H}{\log g}+O(1)
\and
L=\frac{\log R^{2}}{\log g}+O(1)
\]
and so, recalling the definition of $C_{g}$, \cref{lem:hybrid:prefinal} can be bounded as
\[
S
\ll
g^{N}
\exp\biggl(\frac{\log C_{g}}{\log g}\log R^{2}H\biggr)
=
g^{N}
(R^{2}H)^{\alpha_{g}}.
\]
This completes the proof.
\end{proof}

\section{Setting up the discrete circle method}
\label{sec:circle_method}
We now set up the discrete circle method for the proof of \cref{thm:reversed_prime_AP}.
We may assume $N$ to be large in terms of $g$ and $q\ge2$.
By the orthogonality of additive characters, we have
\begin{equation}
\label{circle_method:first_decomp}
\begin{aligned}
\rev{\pi}_{\!\!N}(a,q)
&=
\frac{1}{q}\sum_{0\le k<q}
e\biggl(-\frac{ka}{q}\biggr)
\sum_{p\in\mathscr{G}_{N}}
e\biggl(\frac{k\rev{p}}{q}\biggr)\\
&=
\frac{1}{q}
\sum_{\substack{
0\le k<q\\
q\mid (g^{2}-1)g^{N}k
}}
e\biggl(-\frac{ka}{q}\biggr)
\sum_{p\in\mathscr{G}_{N}}
e\biggl(\frac{k\rev{p}}{q}\biggr)
+
\frac{1}{q}
\sum_{\substack{
0\le k<q\\
q\nmid (g^{2}-1)g^{N}k
}}
e\biggl(-\frac{ka}{q}\biggr)
\sum_{p\in\mathscr{G}_{N}}
e\biggl(\frac{k\rev{p}}{q}\biggr).
\end{aligned}
\end{equation}
The divisibility condition can be rewritten as
\[
q\mid (g^{2}-1)g^{N}k
\iff
\frac{q}{(q,(g^{2}-1)g^{N})}\mathop{\bigg\vert} k.
\]
The very first term of the  right-hand side of \cref{circle_method:first_decomp} is then rewritten as
\begin{align}
&\frac{1}{q}
\sum_{\substack{
0\le k<q\\
q\mid (g^{2}-1)g^{N}k
}}
e\biggl(-\frac{ka}{q}\biggr)
\sum_{p\in\mathscr{G}_{N}}
e\biggl(\frac{k\rev{p}}{q}\biggr)\\
&=
\frac{1}{q}
\sum_{0\le k<(q,(g^{2}-1)g^{N})}
e\biggl(-\frac{ka}{(q,(g^{2}-1)g^{N})}\biggr)
\sum_{p\in\mathscr{G}_{N}}
e\biggl(\frac{k\rev{p}}{(q,(g^{2}-1)g^{N})}\biggr)\\
&=
\frac{(q,(g^{2}-1)g^{N})}{q}
\sum_{\substack{
p\in\mathscr{G}_{N}\\
\rev{p}\equiv a\ \mod{(q,(g^{2}-1)g^{N})}
}}
1
\end{align}
by using the orthogonality backwards. Thus, \cref{circle_method:first_decomp} implies
\[
\rev{\pi}_{\!\!N}(a,q)
=
\frac{(q,(g^{2}-1)g^{N})}{q}
\sum_{\substack{
p\in\mathscr{G}_{N}\\
\rev{p}\equiv a\ \mod{(q,(g^{2}-1)g^{N})}
}}1
+\rev{R}_{N}(a,q)
\]
with
\[
\rev{R}_{N}(a,q)
=
\frac{1}{q}
\sum_{\substack{
0\le k<q\\
q\nmid (g^{2}-1)g^{N}k
}}
e\biggl(-\frac{ka}{q}\biggr)
\sum_{p\in\mathscr{G}_{N}}
e\biggl(\frac{k\rev{p}}{q}\biggr).
\]
Since the exponential sums $F_{N}(\alpha,\beta)$ and $S_{N}(\alpha)$
are extended over integers in $[0,g^{N})$,
again by the orthogonality of additive characters, we have
\begin{equation}
\label{sec:circle_method:revR_expression}
\begin{aligned}
\rev{R}_{N}(a,q)
&=
\frac{1}{qg^{N}}
\sum_{\substack{
0\le k<q\\
q\nmid (g^{2}-1)g^{N}k
}}
e\biggl(-\frac{ka}{q}\biggr)
\sum_{0\le h<g^{N}}
\sum_{\substack{
1\le p<g^{N}
}}
\sum_{n\in\mathscr{G}_{N}}
e\biggl(\frac{k\rev{n}}{q}\biggr)
e\biggl(\frac{h(p-n)}{g^{N}}\biggr)\\
&=
\frac{1}{qg^{N}}
\sum_{\substack{
0\le k<q\\
q\nmid (g^{2}-1)g^{N}k
}}
e\biggl(-\frac{ka}{q}\biggr)
\sum_{0\le h<g^{N}}
S_{N}\biggl(\frac{h}{g^{N}}\biggr)
F_{N}\biggl(-\frac{h}{g^{N}},\frac{k}{q}\biggr).
\end{aligned}
\end{equation}

We now employ the Farey dissection.
Take real parameters $P,Q$ satisfying
\begin{equation}
\label{sec:Farey_dissection:PQ}
4g^{8}\le P\le Q.
\end{equation}
Let us denote the set of the Farey fractions in $[0,1]$ with denominator $\le Q$ by
\[
\mathscr{F}(Q)
\coloneqq
\{
(b,r)\in\mathbb{Z}^{2}
\mid
1\le r\le Q,\ 
0\le b\le r,\ 
(b,r)=1
\}.
\]
For $0\le h<g^{N}$, by Dirichlet's approximation, we can write
\begin{equation}
\label{sec:Farey:h_a_q_eta}
\frac{h}{g^{N}}
=
\frac{b}{r}+\eta
\quad\text{with}\quad
(b,r)\in\mathscr{F}(Q)
\ \text{and}\ 
|\eta|<\frac{1}{rQ}.
\end{equation}
Note that then
\[
\frac{g^{N}b}{r}
+
g^{N}\eta
=
h
\in\mathbb{Z}
\cap\biggl(\frac{g^{N}b}{r}-\frac{g^{N}}{rQ},\frac{g^{N}b}{r}+\frac{g^{N}}{rQ}\biggr)
\cap[0,g^{N}).
\]
The association
\[
\iota
\colon
\{
0\le h<g^{N}
\}
\to
\mathscr{I}
\coloneqq
\left\{
(b,r,\eta)
\midmid
\begin{gathered}
(b,r)\in\mathscr{F}(Q),\ 
|\eta|<\frac{1}{rQ},\ 
\frac{g^{N}b}{r}+g^{N}\eta\in\mathbb{Z}\cap[0,g^{N})
\end{gathered}
\right\}
\]
given by \cref{sec:Farey:h_a_q_eta} is a single-valued function.
Indeed, the value $(b,r,\eta)$ is uniquely determined by $h$
since $\frac{b}{r}$ and other Farey fractions are $\frac{1}{rQ}$ apart and $|\eta|<\frac{1}{rQ}$.
Also, this map is injective since $h$ is determined by $(b,r,\eta)$
by \cref{sec:Farey:h_a_q_eta}.
Let us introduce the major and minor arcs
\begin{align}
\mathscr{I}_{\mathfrak{M}}
&\coloneqq
\{
(b,r,\eta)\in\mathscr{I}
\mid
\max(r,g^{N}|\eta|)\le P
\},\\
\mathscr{I}_{\mathfrak{m}}
&\coloneqq
\{
(b,r,\eta)\in\mathscr{I}
\mid
\max(r,g^{N}|\eta|)>P
\}
\end{align}
so that $\mathscr{I}=\mathscr{I}_{\mathfrak{M}}\sqcup\mathscr{I}_{\mathfrak{m}}$.
We can then decompose \cref{sec:circle_method:revR_expression} as
\begin{equation}
\label{sec:Farey:revR_revRM_revRm}
|\rev{R}_{N}(a,q)|
\le
\rev{R}_{\mathfrak{M}}(q)
+
\rev{R}_{\mathfrak{m}}(q),
\end{equation}
where
\begin{align}
\rev{R}_{\mathfrak{M}}(q)
&\coloneqq
\frac{1}{qg^{N}}
\sum_{\substack{
0\le k<q\\
q\nmid (g^{2}-1)g^{N}k
}}
\sum_{(b,r,\eta)\in\mathscr{I}_{\mathfrak{M}}}
\biggl|
S_{N}\biggl(\frac{b}{r}+\eta\biggr)
F_{N}\biggl(-\biggl(\frac{b}{r}+\eta\biggr),\frac{k}{q}\biggr)
\biggr|,\\
\rev{R}_{\mathfrak{m}}(q)
&\coloneqq
\frac{1}{qg^{N}}
\sum_{\substack{
0\le k<q\\
q\nmid (g^{2}-1)g^{N}k
}}
\sum_{(b,r,\eta)\in\mathscr{I}_{\mathfrak{m}}}
\biggl|
S_{N}\biggl(\frac{b}{r}+\eta\biggr)
F_{N}\biggl(-\biggl(\frac{b}{r}+\eta\biggr),\frac{k}{q}\biggr)
\biggr|.
\end{align}
For a given $(b,r)\in\mathscr{F}(Q)$, we write
\begin{equation}
\mathscr{I}\biggl(\frac{b}{r}\biggr)
\coloneqq
\{
(b,r,\eta)\in\mathscr{I}
\},\quad
\mathscr{I}_{\mathfrak{M}}\biggl(\frac{b}{r}\biggr)
\coloneqq
\{
(b,r,\eta)\in\mathscr{I}_{\mathfrak{M}}
\}
\quad\text{and}\quad
\mathscr{I}_{\mathfrak{m}}\biggl(\frac{b}{r}\biggr)
\coloneqq
\{
(b,r,\eta)\in\mathscr{I}_{\mathfrak{m}}
\}.
\end{equation}

\section{Major arc}
\label{sec:major_arc}
We first bound the major arc contribution $\rev{R}_{\mathfrak{M}}(q)$.

\begin{proposition}
\label{prop:major_measure}
We have $\#\mathscr{I}_{\mathfrak{M}}\ll P^{3}$, where the implicit constant is absolute.
\end{proposition}
\begin{proof}
Since the association $\eta\mapsto\frac{g^{N}b}{r}+g^{N}\eta\eqqcolon h$ is injective, we have
\[
\#\mathscr{I}_{\mathfrak{M}}
\le
\sum_{r\le P}
\sum_{\substack{
0\le b\le r\\
(b,r)=1
}}
\sum_{\substack{
g^{N}|\eta|\le P\\
\frac{g^{N}b}{r}+g^{N}\eta\in\mathbb{Z}
}}
1
\le
\sum_{r\le P}
\sum_{\substack{
0\le b\le r\\
(b,r)=1
}}
\sum_{|h-\frac{g^{N}b}{r}|\le P}
1
\ll
P
\sum_{r\le P}
\sum_{\substack{
0\le b\le r\\
(b,r)=1
}}
1
\ll
P^{3}
\]
as claimed.
\end{proof}

\begin{lemma}
\label{lem:major_arc}
For $N,q\in\mathbb{N}$, we have
\[
\rev{R}_{\mathfrak{M}}(q)
\ll
g^{N}P^{3}
\exp\biggl(-c_{\infty}\cdot\frac{N}{\log q}\biggr)
\]
with some constant $c_{\infty}=c_{\infty}(g)\in(0,1)$,
where the implicit constant depends only on $g$.
\end{lemma}
\begin{proof}
We may assume $N\ge4$.
By \cref{lem:L_infty_bound}, \cref{prop:major_measure} and the trivial estimate for $S_{N}$,
\begin{align}
\rev{R}_{\mathfrak{M}}(q)
&\ll
\frac{1}{g^{N}}
\cdot
P^{3}
\cdot
g^{N}
\cdot
g^{N}\exp\biggl(-c_{\infty}\cdot\frac{N}{\log q}\biggr)
=
g^{N}P^{3}\exp\biggl(-c_{\infty}\cdot\frac{N}{\log q}\biggr)
\end{align}
and so the assertion follows.
\end{proof}

\section{Minor arc}
\label{sec:minor_arc}
We next bound the minor arc contribution $\rev{R}_{\mathfrak{m}}(q)$.

\begin{lemma}
\label{lem:S_bound_pure}
For $b,r\in\mathbb{Z}$ with $r\ge 1$ and $(b,r)=1$
and $\eta\in\mathbb{R}$ with $|\eta|\le r^{-2}$, we have
\[
S_{N}\biggl(\frac{b}{r}+\eta\biggr)
\ll
(
g^{N}r^{-\frac{1}{2}}
+
g^{\frac{4}{5}N}
+
g^{\frac{1}{2}N}r^{\frac{1}{2}}
)
N^{4},
\]
where the implicit constant depends only on $g$.
\end{lemma}
\begin{proof}
Apply partial summation to Theorem~2.1 of \cite[p.~28]{Harman:text}.
\end{proof}

\begin{lemma}
\label{lem:S_bound_eta}
For $b,r\in\mathbb{Z}$ with $r\ge 1$ and $(b,r)=1$
and $\eta\in\mathbb{R}$ with $|\eta|\le r^{-2}$, we have
\[
S_{N}\biggl(\frac{b}{r}+\eta\biggr)
\ll
(
g^{\frac{1}{2}N}(r|\eta|)^{-\frac{1}{2}}
+
g^{\frac{4}{5}N}
+
g^{N}(r|\eta|)^{\frac{1}{2}}
)
N^{4},
\]
where the implicit constant depends only on $g$.
\end{lemma}
\begin{proof}
Apply partial summation to Lemma~4.2 of \cite{Maynard:PnP}.
\end{proof}

\begin{lemma}[Minor arc estimate]
\label{lem:minor_arc}
Let $N\in\mathbb{N}$ and $4g^{8}\le P\le Q\le g^{\frac{1}{2}N-4}$.
Assume that
\begin{equation}
\label{lem:minor_arc:g_cond}
\alpha_{g}<\frac{1}{5}
\quad\text{or, equivalently},\quad
g\ge31699.
\end{equation}
We then have
\[
\rev{R}_{\mathfrak{m}}(q)
\ll
g^{N}
(
P^{2\alpha_{g}-\frac{1}{2}}
+
g^{(\alpha_{g}-\frac{1}{5})N}
+
g^{\alpha_{g}N}Q^{-\frac{1}{2}}
)
N^{6}.
\]
\end{lemma}
\begin{proof}
We may assume $N$ to be large.
By classifying the size of $r$ dyadically, we have
\begin{equation}
\label{lem:minor_arc:Farey}
\rev{R}_{\mathfrak{m}}(q)
\ll
N
\sup_{\beta\in\mathbb{R}}
\sup_{1\le R\le Q}
S(R;\beta),
\end{equation}
where
\begin{equation}
S(R;\beta)
\coloneqq
\frac{1}{g^{N}}
\sum_{R/2\le r\le R}
\astsum_{b\ \mod{r}}
\sum_{(b,r,\eta)\in\mathscr{I}_{\mathfrak{m}}}
\biggl|
S_{N}\biggl(\frac{b}{r}+\eta\biggr)
F_{N}\biggl(-\biggl(\frac{b}{r}+\eta\biggr),\beta\biggr)
\biggr|.
\end{equation}
Note that the summations
\[
\astsum_{b\ \mod{r}}
\and
\sum_{\substack{
0\le b\le r\\
(b,r)=1
}}
\]
do not coincide strictly when $r=1$, but we doubled the first sum to cover the latter sum.
We shall thus bound $S(R;\beta)$ for $1\le R\le Q$.
We further decompose as
\begin{equation}
\label{lem:minor_arc:S_decomp}
S(R;\beta)
\ll
S_{0}(R;\beta)
+
N
\sup_{4g^{8}\le H\le4g^{N}(RQ)^{-1}}
S_{1}(R,H;\beta),
\end{equation}
where
\begin{align}
S_{0}(R;\beta)
&\coloneqq
\frac{1}{g^{N}}
\sum_{R/2\le r\le R}
\astsum_{b\ \mod{r}}
\sum_{\substack{
(b,r,\eta)\in\mathscr{I}_{\mathfrak{m}}\\
g^{N}|\eta|\le 4g^{8}
}}
\biggl|
S_{N}\biggl(\frac{b}{r}+\eta\biggr)
F_{N}\biggl(-\biggl(\frac{b}{r}+\eta\biggr),\beta\biggr)
\biggr|,\\
S_{1}(R,H;\beta)
&\coloneqq
\frac{1}{g^{N}}
\sum_{R/2\le r\le R}
\astsum_{b\ \mod{r}}
\sum_{\substack{
(b,r,\eta)\in\mathscr{I}_{\mathfrak{m}}\\
H/2\le g^{N}|\eta|\le H
}}
\biggl|
S_{N}\biggl(\frac{b}{r}+\eta\biggr)
F_{N}\biggl(-\biggl(\frac{b}{r}+\eta\biggr),\beta\biggr)
\biggr|.
\end{align}
Note that, in order to bound the range of $H$, we used the fact that when $S_{1}(R,H;\beta)\neq 0$ and $R\le Q$,
by taking a term counted in $S_{1}(R,H;\beta)$
and recalling the definition of $\mathscr{I}_{\mathfrak{m}}$, we have
\begin{equation}
\label{lem:minor_arc:S1:RQH1}
H/2\le g^{N}|\eta|\le g^{N}(rQ)^{-1}\le 2g^{N}(RQ)^{-1}
\quad\text{so that}\quad
R^{2}H\le RQH\le 4g^{N}.
\end{equation}
We bound the sums $S_{0}(R;\beta)$ and $S_{1}(R,H;\beta)$ separately.

For the sum $S_{0}(R;\beta)$,
by the definition of $\mathscr{I}_{\mathfrak{m}}$ and \cref{sec:Farey_dissection:PQ},
we have
\[
S_{0}(R;\beta)=0
\quad\text{if $R\le P$}.
\]
We may thus assume $P<R\le Q$. By the assumption $Q\le g^{\frac{1}{2}N-4}$,
the choice $H=4g^{8}$ gives
\[
R^{2}H
\le
4g^{8}Q^{2}
=
4g^{N}.
\]
We can thus use \cref{lem:hybrid} and \cref{lem:S_bound_pure} to get
\begin{align}
S_{0}(R;\beta)
&\ll
\frac{1}{g^{N}}
(
g^{N}R^{-\frac{1}{2}}
+
g^{\frac{4}{5}N}
+
g^{\frac{1}{2}N}R^{\frac{1}{2}}
)
N^{4}
\times
\sum_{r\le R}
\astsum_{b\ \mod{r}}
\sum_{\substack{
g^{N}|\eta|\le 4g^{8}\\
\frac{g^{N}b}{r}+g^{N}\eta\in\mathbb{Z}
}}
\biggl|
F_{N}\biggl(-\biggl(\frac{b}{r}+\eta\biggr),\beta\biggr)
\biggr|\\
&\ll
g^{N}
(
R^{2\alpha_{g}-\frac{1}{2}}
+
g^{-\frac{1}{5}N}R^{2\alpha_{g}}
+
g^{-\frac{1}{2}N}R^{2\alpha_{g}+\frac{1}{2}}
)
N^{4}.
\end{align}
Thus, by recalling $P<R\le Q$ and using \cref{lem:minor_arc:g_cond}, we have
\begin{equation}
\label{lem:minor_arc:S0}
S_{0}(R;\beta)
\ll
g^{N}
(
P^{2\alpha_{g}-\frac{1}{2}}
+
g^{-\frac{1}{5}N}Q^{2\alpha_{g}}
+
g^{-\frac{1}{2}N}Q^{2\alpha_{g}+\frac{1}{2}}
)
N^{4}.
\end{equation}
This completes the estimate of $S_{0}(R;\beta)$.

We next consider the sum $S_{1}(R,H;\beta)$.
When $S_{1}(R,H;\beta)\neq 0$ and $R\le Q$,
by taking a term counted in $S_{1}(R,H;\beta)$
and recalling the definition of $\mathscr{I}_{\mathfrak{m}}$,
we have
\begin{equation}
\label{lem:minor_arc:S1:RQH2}
R\le P
\implies
P<g^{N}|\eta|\le H
\end{equation}
besides \cref{lem:minor_arc:S1:RQH1}.
We may assume $R\ge1$ and $H\ge 4g^{8}$.
By \cref{lem:hybrid} and \cref{lem:S_bound_eta}, we then get
\begin{equation}
\label{lem:minor_arc:S1:general}
\begin{aligned}
S_{1}(R,H;\beta)
&\ll
\frac{1}{g^{N}}
(
g^{N}(RH)^{-\frac{1}{2}}
+
g^{\frac{4}{5}N}
+
g^{\frac{1}{2}N}(RH)^{\frac{1}{2}}
)
N^{4}\\
&\hspace{0.1\textwidth}
\times
\sum_{r\le R}
\astsum_{b\ \mod{r}}
\sum_{\substack{
g^{N}|\eta|\le H\\
\frac{g^{N}b}{r}+g^{N}\eta\in\mathbb{Z}
}}
\biggl|
F_{N}\biggl(-\biggl(\frac{b}{r}+\eta\biggr),\beta\biggr)
\biggr|\\
&\ll
g^{N}
(
(R^{2}H)^{\alpha_{g}}
(RH)^{-\frac{1}{2}}
+
g^{-\frac{1}{5}N}(R^{2}H)^{\alpha_{g}}
+
g^{-\frac{1}{2}N}
(R^{2}H)^{\alpha_{g}}
(RH)^{\frac{1}{2}}
)
N^{4}.
\end{aligned}
\end{equation}
When $R\le P$,
since $P\le H\le 4g^{N}(RQ)^{-1}$,
by using \cref{lem:minor_arc:g_cond}, \cref{lem:minor_arc:S1:RQH1} and \cref{lem:minor_arc:S1:RQH2}
in \cref{lem:minor_arc:S1:general}, we have
\begin{equation}
\begin{aligned}
S_{1}(R,H;\beta)
\ll
g^{N}
(
R^{2\alpha_{g}-\frac{1}{2}}P^{\alpha_{g}-\frac{1}{2}}
+
g^{-\frac{1}{5}N}(g^{N}RQ^{-1})^{\alpha_{g}}
+
g^{-\frac{1}{2}N}
(g^{N}RQ^{-1})^{\alpha_{g}}
(g^{N}Q^{-1})^{\frac{1}{2}}
)
N^{4}
\end{aligned}
\end{equation}
and so
\begin{equation}
\label{lem:minor_arc:S1:small_P}
\begin{aligned}
S_{1}(R,H;\beta)
\ll
g^{N}
(
P^{\alpha_{g}-\frac{1}{2}}
+
g^{(\alpha_{g}-\frac{1}{5})N}P^{\alpha_{g}}Q^{-\alpha_{g}}
+
g^{\alpha_{g}N}
P^{\alpha_{g}}
Q^{-(\alpha_{g}+\frac{1}{2})}
)
N^{4}
\quad\text{if $1\le R\le P$}.
\end{aligned}
\end{equation}
When $R>P$, since $4g^{8}\le H\le 4g^{N}(RQ)^{-1}$ by \cref{lem:minor_arc:S1:RQH1},
by using \cref{lem:minor_arc:g_cond},
we can bound \cref{lem:minor_arc:S1:general} as
\begin{equation}
\label{lem:minor_arc:S1:large_P}
\begin{aligned}
S_{1}(R,H;\beta)
&\ll
g^{N}
(
R^{2\alpha_{g}-\frac{1}{2}}
+
g^{-\frac{1}{5}N}(g^{N}RQ^{-1})^{\alpha_{g}}
+
g^{-\frac{1}{2}N}
(g^{N}RQ^{-1})^{\alpha_{g}}
(g^{N}Q^{-1})^{\frac{1}{2}}
)
N^{4}\\
&\ll
g^{N}
(
P^{2\alpha_{g}-\frac{1}{2}}
+
g^{(\alpha_{g}-\frac{1}{5})N}
+
g^{\alpha_{g}N}Q^{-\frac{1}{2}}
)
N^{4}
\quad\text{if $P<R\le Q$}.
\end{aligned}
\end{equation}
By combining \cref{lem:minor_arc:S1:small_P} and \cref{lem:minor_arc:S1:large_P} and noting that
\[
P^{\alpha_{g}-\frac{1}{2}}
\le
P^{2\alpha_{g}-\frac{1}{2}},\quad
g^{(\alpha_{g}-\frac{1}{5})N}P^{\alpha_{g}}Q^{-\alpha_{g}}
\le
g^{(\alpha_{g}-\frac{1}{5})N},\quad
g^{\alpha_{g}N}
P^{\alpha_{g}}
Q^{-(\alpha_{g}+\frac{1}{2})}
\le
g^{\alpha_{g}N}Q^{-\frac{1}{2}},
\]
we obtain
\begin{equation}
\label{lem:minor_arc:S1}
S_{1}(R,H;\beta)
\ll
g^{N}
(
P^{2\alpha_{g}-\frac{1}{2}}
+
g^{(\alpha_{g}-\frac{1}{5})N}
+
g^{\alpha_{g}N}Q^{-\frac{1}{2}}
)
N^{4}.
\end{equation}
So \(S_{1}(R,H;\beta)\) can be bounded similarly
for $R>P$ and $R\le P$.

On inserting \cref{lem:minor_arc:S0} and \cref{lem:minor_arc:S1}
into \cref{lem:minor_arc:S_decomp} by noting that
\[
g^{-\frac{1}{5}N}Q^{2\alpha_{g}}
\le
g^{(\alpha_{g}-\frac{1}{5})N}
\and
g^{-\frac{1}{2}N}
Q^{2\alpha_{g}+\frac{1}{2}}
\le
g^{(\alpha_{g}-\frac{1}{4})N}
\le
g^{(\alpha_{g}-\frac{1}{5})N},
\]
we have
\begin{equation}
S(R;\beta)
\ll
g^{N}
(
P^{2\alpha_{g}-\frac{1}{2}}
+
g^{(\alpha_{g}-\frac{1}{5})N}
+
g^{\alpha_{g}N}Q^{-\frac{1}{2}}
)
N^{5}.
\end{equation}
On inserting this estimate into \cref{lem:minor_arc:Farey}, we obtain the assertion.
\end{proof}

\begin{remark}
\label{rem:alpha_g_monotonic}
We can check $\alpha_{g}$ is decreasing for $g\ge9$ as follows. Write
\[
C(g)
\coloneqq
C_{g}
=
\frac{2}{\pi}\log\cot\frac{\pi}{2g}+\frac{1}{g\sin\frac{\pi}{2g}}+1.
\]
Then, we have
\[
\alpha_{g}
=
\frac{\log C(g)}{\log g}
\quad\text{and so}\quad
\frac{d\alpha_{g}}{dg}
=
\biggl(\frac{C'(g)}{C(g)}-\frac{\log C(g)}{g\log g}\biggr)\frac{1}{\log g}.
\]
For $g\ge9$, we have
\[
\log C(g)
\ge
\log\biggl(\frac{2}{\pi}\log\cot\frac{\pi}{2g}+\frac{1}{g\sin\frac{\pi}{2g}}+1\biggr)
\ge
\log\biggl(\frac{2}{\pi}\log\cot\frac{\pi}{18}+\frac{2}{\pi}+1\biggr)
=
1.008\ldots
>1
\]
and so it suffices to show
\begin{equation}
\label{rem:alpha_g_monotonic:goal}
C'(g)g\log g\le C(g)
\end{equation}
for $g\ge 9$. By using $\cot x\le\frac{1}{x}$ for $x\in[0,\frac{\pi}{2}]$, We have
\[
C'(g)g\log g
=
\biggl(
\frac{1}{g^{2}\cot\frac{\pi}{2g}\sin\frac{\pi}{2g}}
-
\biggl(1-\frac{\pi}{2g}\cot\frac{\pi}{2g}\biggr)
\frac{1}{g^{2}\sin\frac{\pi}{2g}}
\biggr)g\log g
\le
\frac{2\log g}{g\sin\frac{\pi}{g}}.
\]
By recalling $(\frac{\pi}{2}\log\cot\frac{\pi}{2}x)'=-\frac{2}{\sin\pi x}$, we then have
\[
C(g)
\ge
\frac{2}{\pi}\log\cot\frac{\pi}{2g}+1
=
\frac{2}{\pi}\log\cot\frac{\pi}{2g}-\frac{2}{\pi}\log\cot\frac{\pi}{4}+1
=
\int_{\frac{1}{g}}^{\frac{1}{2}}\frac{2x}{\sin\pi x}\frac{dx}{x}+1.
\]
Since $(\frac{\sin x}{x})=\frac{x-\tan x}{x^{2}}\cdot\cos x\le 0$
and $\sin x\ge\frac{2}{\pi}x$ for $x\in[0,\frac{\pi}{4}]$, we have
\[
C(g)
\ge
\frac{2\log g}{g\sin\frac{\pi}{g}}
+
1
-
\frac{2\log 2}{g\sin\frac{\pi}{g}}
\ge
\frac{2\log g}{g\sin\frac{\pi}{g}}
+
1
-
\log 2
\ge
\frac{2\log g}{g\sin\frac{\pi}{g}}
\ge
C'(g)g\log g
\]
as desired. We can check $\alpha_{31699}=0.1999997\ldots<\frac{1}{5}$ by some numerical calculation.
\end{remark}

\section{Proof of the main theorem and corollaries}
\label{sec:prf_main_thm}
We finally prove the main theorem and its corollaries.

\begin{proof}[Proof of \cref{thm:reversed_prime_AP}.]
In the above setting of the discrete circle method, we choose $P,Q$ by
\begin{equation}
\label{thm:reversed_prime_AP:PQ}
P
\coloneqq
\max\biggl(\exp\biggl(\frac{c_{\infty}}{4}\frac{N}{\log q}\biggr),4g^{8}\biggr)
\and
Q\coloneqq g^{\frac{1}{2}N-4},
\end{equation}
where $c_{\infty}\in(0,1)$ is a constant in \cref{lem:major_arc}.
By \cref{lem:major_arc}, we then have
\begin{equation}
\label{thm:reversed_prime_AP:RM}
\rev{R}_{\mathfrak{M}}(q)
\ll
g^{N}\exp\biggl(-\frac{c_{\infty}}{4}\frac{N}{\log q}\biggr).
\end{equation}
Also, by \cref{lem:minor_arc},
 noting that \cref{thm:reversed_prime_AP:g_cond} and \cref{thm:reversed_prime_AP:PQ}
assure the required conditions, we have
\begin{equation}
\rev{R}_{\mathfrak{m}}(q)
\ll
g^{N}
(
P^{2(\alpha_{g}-\frac{1}{4})}
+
g^{(\alpha_{g}-\frac{1}{5})N}
)
N^{6}
\ll
g^{N}N^{6}
\exp\biggl(-\widetilde{c}\cdot\frac{N}{\log q}\biggr)
\end{equation}
with some $\widetilde{c}=\widetilde{c}(g)\in(0,c_{\infty}]$.
Assuming
\begin{equation}
\label{thm:reversed_prime_AP:q_cond_in_proof}
q\le\exp\biggl(c\cdot\frac{N}{\log N}\biggr)
\quad\text{with}\quad
c\coloneqq\frac{\widetilde{c}}{12}\in\biggl[0,\frac{c_{\infty}}{4}\biggr],
\end{equation}
we have
\begin{align}
N^{6}
\exp\biggl(-\widetilde{c}\cdot\frac{N}{\log q}\biggr)
=
\exp\biggl(-\widetilde{c}\cdot\frac{N}{\log q}+6\log N\biggr)
\le
\exp\biggl(-\frac{\widetilde{c}}{2}\cdot\frac{N}{\log q}\biggr)
\le
\exp\biggl(-c\cdot\frac{N}{\log q}\biggr)
\end{align}
and so
\begin{equation}
\label{thm:reversed_prime_AP:Rm}
\rev{R}_{\mathfrak{m}}(q)
\ll
g^{N}
\exp\biggl(-c\cdot\frac{N}{\log q}\biggr).
\end{equation}
On inserting \cref{thm:reversed_prime_AP:RM} and \cref{thm:reversed_prime_AP:Rm}
into \cref{sec:Farey:revR_revRM_revRm}, we arrive at
\[
\rev{R}_{N}(a,q)
\ll
g^{N}
\exp\biggl(-c\cdot\frac{N}{\log q}\biggr)
\]
provided \cref{thm:reversed_prime_AP:q_cond_in_proof} is true. This completes the proof.
\end{proof}

\begin{proof}[Proof of \cref{cor:ZsiflawLegeis}.]
The case $q=1$ follows immediately by the prime number theorem and so we may assume $q\ge2$.
Since \cref{cor:ZsiflawLegeis:d_cond} is stronger than \cref{thm:reversed_prime_AP:d_cond},
we can use the conclusion of \cref{thm:reversed_prime_AP}.
Then, the remaining task is to evaluate the cardinality
\[
G
\coloneqq
\sum_{\substack{
p\in\mathscr{G}_{N}\\
\rev{p}\equiv a\ \mod{(q,(g^{2}-1)g^{N})}
}}
1.
\]
Note that the above congruence can be rephrased by the simultaneous congruences
\[
\rev{p}\equiv a\ \mod{(q,(g^{2}-1)g^{N})}
\iff
\left\{
\begin{aligned}
\rev{p}&\equiv a\ \mod{(q,g^{2}-1)},\\
\rev{p}&\equiv a\ \mod{(q,g^{N})}.
\end{aligned}
\right.
\]
by the Chinese remainder theorem since $(g^{2}-1,g)=1$.
As stated in \cref{rev_congruence}, we have
\begin{equation}
\label{cor:ZsiflawLegeis:g_pm_1_cong_rev}
\rev{p}\equiv a\ \mod{(q,g^{2}-1)}
\iff
p\equiv \overline{g}^{N-1}a\ \mod{(q,g^{2}-1)},
\end{equation}
where $\overline{g}\ \mod{(q,g^{2}-1)}$ is the multiplicative inverse of $g\ \mod{(q,g^{2}-1)}$,
since by using $g^{2}\equiv1\ \mod{g^{2}-1}$, we can see that the base-$g$ representation
\begin{equation}
\label{cor:ZsiflawLegeis:p_expansion}
p=\sum_{0\le i<N}p_{i}g^{i}
\quad\text{with}\quad
p_{0},\ldots,p_{N-1}\in\{0,\ldots,g-1\}
\ \text{and}\ 
p_{0},p_{N-1}\neq0
\end{equation}
of $p\in\mathscr{G}_{N}$ satisfies
\begin{align}
\rev{p}
&\equiv
\sum_{0\le i<N}p_{i}g^{N-i-1}
\equiv
g^{N-1}\sum_{0\le i<N}p_{i}g^{i}(g^{2})^{-i}
\equiv
g^{N-1}
\sum_{0\le i<N}p_{i}g^{i}
\equiv
g^{N-1}p\ \mod{g^{2}-1}.
\end{align}
We can thus rewrite $G$ as
\[
G
=
\sum_{\substack{
p\in\mathscr{G}_{N}\\
p\equiv\overline{g}^{N-1}a\ \mod{(q,g^{2}-1)}\\
\rev{p}\equiv a\ \mod{(q,g^{N})}
}}
1.
\]
When $(a,q,g^{2}-1)\neq1$, we have
\[
G
\le
\sum_{p\mid (a,q,g^{2}-1)}
1
\ll
1.
\]
We  may thus assume $(a,q,g^{2}-1)=1$.
Take the smallest $N_{0}\in\mathbb{N}$ such that
\[
(q,g^{N_{0}})=(q,g^{N}).
\]
By noting that $\rev{n}\not\equiv0\ \mod{g}$ for any $n\in\mathbb{N}$, we then have
\[
G
=
\sum_{\substack{
1\le v<g^{N_{0}}\\
v\equiv a\ \mod{(q,g^{N_{0}})}\\
v\not\equiv0\ \mod{g}
}}
\sum_{\substack{
p\in\mathscr{G}_{N}\\
p\equiv\overline{g}^{N-1}a\ \mod{(q,g^{2}-1)}\\
\rev{p}\equiv v\ \mod{g^{N_0}}
}}
1.
\]
When $g\mid(a,q)$, the above sum over $v$ is empty and so we may assume $g\nmid (a,q)$.
For the base-$g$ expansion \cref{cor:ZsiflawLegeis:p_expansion} of $p\in\mathscr{G}_{N}$,
we have
\[
\rev{p}\equiv v\ \mod{g^{N_0}}
\iff
\sum_{0\le i<N_{0}}p_{N-i-1}g^{i}
\equiv
v\ \mod{g^{N_{0}}}
\]
However, since $v$ and $\sum_{0\le i<N_{0}}p_{N-i-1}g^{i}$ both belong to $[0,g^{N_{0}})$, this is further equivalent to
\[
\rev{p}\equiv v\ \mod{g^{N_0}}
\iff
v
=\sum_{0\le i<N_{0}}p_{N-i-1}g^{i}
\iff
v_{i}=p_{N-i-1},
\]
where the $v_{i}$ are given by the base-$g$ representation of $v$ in the form
\[
v=\sum_{0\le i<N_{0}}v_{i}g^{i}
\quad\text{with}\quad
v_{0},\ldots,v_{N_{0}-1}\in\{0,\ldots,g-1\}
\ \text{and}\ 
v_{0}\neq0.
\]
Thus, we have $\rev{p}\equiv v\ \mod{g^{N_0}}$ and $p\in\mathscr{G}_{N}$ if and only if
\[
\sum_{N-N_{0}\le i<N}p_{i}g^{i}
=
\sum_{0\le i<N_{0}}p_{N-i-1}g^{N-i-1}
=
\sum_{0\le i<N_{0}}v_{i}g^{N-i-1}
\eqqcolon
v^{\ast}.
\]
This is further equivalent to
\[
p\in[v^{\ast},v^{\ast}+g^{N-N_{0}})
\]
since $v_{0}\neq0$ implies $v^{\ast}\ge g^{N-1}$.
We can thus further rewrite $G$ as
\[
G
=
\sum_{\substack{
1\le v<g^{N_{0}}\\
v\equiv u\ \mod{(q,g^{N_{0}})}\\
v\not\equiv0\ \mod{g}
}}
\sum_{\substack{
p\in[v^{\ast},v^{\ast}+g^{N-N_{0}})\\
p\equiv\overline{g}^{N-1}a\ \mod{(q,g^{2}-1)}
}}
1.
\]
By using some effective prime number theorem in arithmetic progressions with modulus $(q,g^{2}-1)\ll1$,
e.g.\ Corollary~11.12 and Theorem~11.17 of \cite{MV:text},
and using the approximation
\[
\int_{v^{\ast}}^{v^{\ast}+g^{N-N_{0}}}
\frac{dt}{\log t}
=
\frac{g^{N-N_{0}}}{\log g^{N}}
\biggl(1+O\biggl(\frac{1}{N}\biggr)\biggr)
\]
for $1\le v<N_{0}$ with $v\not\equiv0\ \mod{g}$, we have
\[
G
=
\frac{1}{\phi((q,g^{2}-1))}
\biggl(\sum_{\substack{
1\le v<g^{N_{0}}\\
v\equiv u\ \mod{(q,g^{N_{0}})}\\
v\not\equiv0\ \mod{g}
}}
1\biggr)
\frac{g^{N-N_{0}}}{\log g^{N}}
\biggl(1+O\biggl(\frac{1}{N}\biggr)\biggr)
+
O(g^{N+N_{0}}\exp(-c_{\textrm{PNT}}\sqrt{N})).
\]
Since
\begin{align}
\sum_{\substack{
1\le v<g^{N_{0}}\\
v\equiv a\ \mod{(q,g^{N_{0}})}\\
v\not\equiv0\ \mod{g}
}}
1
=
\sum_{\substack{
v\ \mod{g^{N_{0}}}\\
v\equiv a\ \mod{(q,g^{N_{0}})}\\
v\not\equiv0\ \mod{g}
}}
1
&=
\sum_{\substack{
v\ \mod{g^{N_{0}}}\\
v\equiv a\ \mod{(q,g^{N_{0}})}
}}
1
-
\sum_{\substack{
v\ \mod{g^{N_{0}}}\\
v\equiv a\ \mod{(q,g^{N_{0}})}\\
v\equiv0\ \mod{g}
}}
1\\
&=
\frac{g^{N_{0}}}{(q,g^{N_{0}})}
-
\sum_{\substack{
v\ \mod{g^{N_{0}-1}}\\
gv\equiv a\ \mod{(q,g^{N_{0}})}
}}
1\\
&=
\biggl(
1
-
\mathbbm{1}_{(q,g)\mid a}
\frac{(q,g)}{g}
\biggr)
\frac{g^{N_{0}}}{(q,g^{N_{0}})}\\
&=
\biggl(
1
-
\mathbbm{1}_{(q,g)\mid a}
\frac{(q,g)}{g}
\biggr)
\frac{g^{N_{0}}}{(q,g^{N})},
\end{align}
we obtain
\begin{equation}
\label{cor:ZsiflawLegeis:prefinal}
G
=
\frac{\rho_{g}(a,q)}{(q,(g^{2}-1)g^{N})}
\frac{g^{N}}{\log g^{N}}
\biggl(1+O\biggl(\frac{1}{N}\biggr)\biggr)
+
O(g^{N+N_{0}}\exp(-c_{\textrm{PNT}}\sqrt{N})).
\end{equation}
We then estimate $g^{N_{0}}$ in the remainder term.
For a prime number $p$ and a positive integer $n$,
define $v_{p}(n)\in\mathbb{Z}$ by $p^{v_{p}(n)}\mid n$ but $p^{v_{p}(n)+1}\nmid n$.
For any prime $p$, by \cref{cor:ZsiflawLegeis:d_cond}, we then have
\[
2^{v_{p}(q)}
\le
p^{v_{p}(q)}
\le
q
\le
\exp(c\sqrt{N})
\quad\text{and so}\quad
v_{p}(q)\le\frac{c}{\log 2}\sqrt{N}\le N
\]
by making $c<\log 2$. Then, we have
\[
(q,g^{[\frac{c}{\log 2}\sqrt{N}]})
=
\prod_{p\mid(q,g)}p^{\min(v_{p}(q),v_{p}(g)[\frac{c}{\log 2}\sqrt{N}])}
=
\prod_{p\mid(q,g)}p^{v_{p}(q)}
=
\prod_{p\mid(q,g)}p^{\min(v_{p}(q),v_{p}(g)N)}
=
(q,g^{N})
\]
and so $N_{0}\le\frac{c}{\log 2}\sqrt{N}$ by the minimality of $N_{0}$.
By making $c$ small enough so that
\[
g^{\frac{c}{\log 2}\sqrt{N}}
\le
\exp(\tfrac{1}{2}c_{\textrm{PNT}}\sqrt{N})
\quad\text{and}\quad
c<\tfrac{1}{2}c_{\textrm{PNT}},
\]
we then have
\[
g^{N+N_{0}}\exp(-c_{\textrm{PNT}}\sqrt{N})
\le
g^{N}\exp(-\tfrac{1}{2}c_{\textrm{PNT}}\sqrt{N})
\le
g^{N}\exp(-c\sqrt{N}).
\]
On inserting this into \cref{cor:ZsiflawLegeis:prefinal}
and combining it with \cref{thm:reversed_prime_AP},
we obtain the assertion.
\end{proof}

\begin{proof}[Proof of \cref{cor:Telhcirid}.]
Assume $(a,q,g^{2}-1)=1$ and $g\nmid(a,q)$.
If $(q,g)\mid a$, we cannot have $g\mid (q,g)$ since otherwise $g\mid (a,q)$.
We thus have
\[
\rho_{g}(a,q)
=
\biggl(
1
-
\mathbbm{1}_{(q,g)\mid a}
\frac{(q,g)}{g}
\biggr)
\prod_{p\mid (q,g^{2}-1)}\biggl(\frac{p}{p-1}\biggr)
\ge
\frac{1}{g}.
\]
Thus, the assertion follows by \cref{cor:ZsiflawLegeis}.
\end{proof}

\section*{Acknowledgments}
The first author would like to thank Lucile Devin, Didier Lesesvre,
Bruno Martin, Thi-Thu Nguyen and Martine Queffélec 
with whom she participated in the 2022--23
Lille-Calais {\it Groupe de Travail}  
on Maynard's paper \cite{Maynard:PnP}.
The second author would like to express his gratitude to Laboratoire Paul Painlev\'{e}
and Prof.~Daniel Duverney for their generous hospitality  in Lille
in October--November 2023 during which the main part of this work 
was carried out.
The second author was supported by JSPS KAKENHI Grant Number JP21K13772.

\bibliographystyle{amsplain}
\bibliography{Telhcirid}
\bigskip

\noindent
{\textsc{%
\small
Gautami Bhowmik\\[.3em]
\footnotesize
Laboratoire Paul Painlev\'{e}, Labex-CEMPI, Universit\'{e} de Lille\\
59655 Villeneuve d'Ascq Cedex, France.
}

\noindent
\small
\textit{Email address}: \texttt{gautami.bhowmik@univ-lille.fr}
}
\bigskip

\noindent
{\textsc{%
\small
Yuta Suzuki\\[.3em]
\footnotesize
Department of Mathematics, Rikkyo University,\\
3-34-1 Nishi-Ikebukuro, Toshima-ku, Tokyo 171-8501, Japan.
}

\noindent
\small
\textit{Email address}: \texttt{suzuyu@rikkyo.ac.jp}
}

\end{document}